\documentclass{amsart}
\usepackage{amsmath, amsthm, amssymb, amscd, epsfig}
\usepackage{pinlabel}

\begin{document}

\newtheorem{theorem}{Theorem}[section]
\newtheorem{lemma}[theorem]{Lemma}
\newtheorem{proposition}[theorem]{Proposition}
\newtheorem{corollary}[theorem]{Corollary}
\newtheorem{conjecture}[theorem]{Conjecture}
\newtheorem{question}[theorem]{Question}
\newtheorem{problem}[theorem]{Problem}
\newtheorem*{claim}{Claim}
\newtheorem*{criterion}{Criterion}
\newtheorem*{isometry_thm}{Random Isometry Theorem~\ref{isometry_theorem}}
\newtheorem*{isometry_con}{Isometry Conjecture~\ref{isometry_conjecture}}
\newtheorem*{fatgraph_thm}{Random Fatgraph Theorem~\ref{fatgraph_theorem}}

\theoremstyle{definition}
\newtheorem{definition}[theorem]{Definition}
\newtheorem{construction}[theorem]{Construction}
\newtheorem{notation}[theorem]{Notation}

\theoremstyle{remark}
\newtheorem{remark}[theorem]{Remark}
\newtheorem{example}[theorem]{Example}

\numberwithin{equation}{subsection}

\def\line{\hbox{{\vrule height 5.1pt}}}

\def\Z{\mathbb Z}
\def\R{\mathbb R}
\def\Q{\mathbb Q}
\def\H{\mathbb H}
\def\E{\mathcal E}
\def\M{\mathcal M}
\def\C{\mathcal C}
\def\SS{\mathcal S}

\def\freq{\textnormal{freq}}
\def\cl{\textnormal{cl}}
\def\scl{\textnormal{scl}}
\def\homeo{\textnormal{Homeo}}
\def\rot{\textnormal{rot}}
\def\area{\textnormal{area}}
\def\Aut{\textnormal{Aut}}
\def\Out{\textnormal{Out}}
\def\MCG{\textnormal{MCG}}

\def\Id{\textnormal{Id}}
\def\PSL{\textnormal{PSL}}
\def\til{\widetilde}

\title{Isometric endomorphisms of free groups}
\author{Danny Calegari}
\address{Department of Mathematics \\ Caltech \\
Pasadena CA, 91125}
\email{dannyc@its.caltech.edu}
\author{Alden Walker}
\address{Department of Mathematics \\ Caltech \\
Pasadena CA, 91125}
\email{awalker@caltech.edu}
\date{version 0.15, \today}

\begin{abstract}
An arbitrary homomorphism between groups is nonincreasing for stable commutator 
length, and there are infinitely many (injective) homomorphisms between free groups 
which strictly decrease the stable commutator length of some elements. 
However, we show in this paper that a {\em random} homomorphism between free groups 
is almost surely an isometry for stable commutator length for every element; 
in particular, the unit ball in the scl norm of a free group admits an enormous 
number of {\em exotic isometries}. 

Using similar methods, we show that a 
random fatgraph in a free group is extremal (i.e. is an absolute minimizer 
for relative Gromov norm) for its boundary; this implies, for instance, 
that a random element of a free group with commutator length at most $n$ has 
commutator length exactly $n$ and stable commutator length exactly $n-1/2$.
Our methods also let us construct explicit (and computable) quasimorphisms
which certify these facts.
\end{abstract}

\maketitle

\section{Introduction}

\subsection{Stable commutator length}

If $G$ is a group, the {\em commutator length} $\cl(g)$ of an element $g \in G'$ is the
least number of commutators in $G$ whose product is $g$, and the {\em stable commutator
length} is the limit $\lim_{n \to \infty} \cl(g^n)/n$. Stable commutator length $\scl$
extends to a pseudo-norm on the space $B_1(G)$ of formal real (group) 1-boundaries,
and descends to a further quotient $B_1^H(G):=B_1(G)/\langle g^n-ng, g-hgh^{-1}\rangle$,
reflecting the fact that $\scl$ is homogeneous (by definition), and a class function.
When $G$ is hyperbolic, $\scl$ is a {\em norm} on $B_1^H(G)$ (\cite{Calegari_Fujiwara}, Thm.~A${}^\prime$). 
The crucial properties of this pseudo-norm in general are
\begin{enumerate}
\item{{\bf (characteristic)} it is constant on orbits of $\Out(G)$; and}
\item{{\bf (monotone)} it is nonincreasing under homomorphisms between groups.}
\end{enumerate}
Of course the first property follows from the second.

\subsection{Exotic isometries}

If $G$ admits a large group of automorphisms, the characteristic property becomes very
interesting. Perhaps the most interesting example is the case of a free group $F$; in
this case, we obtain a natural isometric action of $\Out(F)$ on the normed space
$B_1^H(F)$. In fact, the unit ball
in the $\scl$ norm on $B_1^H(F)$ is a {\em polyhedron}, and associated to every realization
of $F$ as $\pi_1(S)$ for $S$ a compact, oriented surface, there is a top dimensional
face $\pi_S$ of the unit ball whose stabilizer in $\Out(F)$ is precisely the
mapping class group $\MCG(S)$; see \cite{Calegari_rational, Calegari_faces} for
proofs of these facts.

There are many natural realizations of $\MCG(S)$ and
$\Out(F)$ as groups of isometries of geometric spaces. Inevitably, these spaces admit
essentially no other isometries (up to finite index). For example, in the case of $\MCG(S)$
acting on Teichm\"uller space, this is a famous theorem of Royden \cite{Royden}. In marked
contrast to these examples, our first main result is that the $\scl$ unit ball in $B_1^H(F)$ admits an
enormous number of {\em exotic isometries}, and in fact we show that a {\em random} homomorphism
between free groups is almost surely an isometry for stable commutator length:

\begin{isometry_thm}
A random homomorphism $\varphi:F_k \to F_l$ of length $n$
between free groups of ranks $k,l$ is an isometry of $\scl$ with probability $1-O(C(k,l)^{-n})$
for some constant $C(k,l)>1$.
\end{isometry_thm}
Here a random homomorphism of length $n$ is one which sends the generators of $F_k$ to
randomly chosen elements of $F_l$ of length at most $n$.

We remark that in \cite{Bestvina_Feighn} (Lem.~6.1) Bestvina--Feighn obtained partial results in the
direction of this theorem. Explicitly, for any element $w$ in a free group $F$, and for any
other free group $F'$, they constructed many homomorphisms $\varphi:F \to F'$ for which
the {\em commutator length} ({\em not} the stable commutator length) of $\varphi(w)$ in $F'$
is equal to the commutator length of $w$ in $F$.
In fact, their technique implies (though they do not state this explicitly)
that for each fixed $w$, a random homomorphism of length $n$ has this property with
probability $1-O(C(w)^{-n})$. However the constant $C(w)$ they obtain definitely depends on $w$,
and therefore they do not exhibit a single homomorphism which is an isometry for commutator
length for all $w$ simultaneously (in fact, our proof of the Isometry Theorem should be valid
with $\scl$ replaced by $\cl$, but we have not pursued this). 

A necessary condition for a homomorphism between free groups to be an isometry for $\scl$ is for
it to be injective. However, if $k\ge 3$ then there are many injective homomorphisms
$F_k \to F_l$ that are not isometries; we give two infinite classes of examples, namely
Example~\ref{curve_complex_example} and Example~\ref{nongeometric_cover_example}. 
In fact, we show (Proposition~\ref{prop:self_commensurating})
that if $F_k \to F_l$ is an isometry, then the image of $F_k$ is
necessarily {\em self-commensurating} in $F_l$; i.e.\/ it is not properly contained with
finite index in any other subgroup. Of course, any injective homomorphism $F_2 \to F_l$ has
self-commensurating image. Extensive computer evidence (and some theory) has 
led us to make the following conjecture:

\begin{isometry_con}
Let $\varphi:F_2 \to F$ be any injective homomorphism from a free group of rank $2$ to a free group $F$.
Then $\varphi$ is an isometry of $\scl$. 
\end{isometry_con}

\subsection{Extremal fatgraphs and quasimorphisms}

There is a duality theorem (Generalized Bavard duality; see \cite{Calegari_rational} or \cite{Calegari_scl}
Thm.~2.79; also see \cite{Bavard}) 
relating stable commutator length to an important class of functions called
{\em homogeneous quasimorphisms}. If $G$ is a group, a function $\phi:G \to \R$ is a
homogeneous quasimorphism if it satisfies $\phi(g^n)=n\phi(g)$ for every $g\in G$, and if there is
a least non-negative number $D(\phi)$ (called the {\em defect}) so that for all $g,h \in G$, there is
an inequality
$$|\phi(gh)-\phi(g)-\phi(h)|\le D(\phi)$$
The space of homogeneous quasimorphisms on $G$ is a vector space $Q(G)$. The subspace on which
$D$ vanishes is naturally isomorphic to $H^1(G;\R)$, and $D$ defines a norm on $Q/H^1$ making it
into a {\em Banach space}.

Generalized Bavard duality is the statement that for all chains 
$\sum t_i g_i \in B_1^H(G)$ there is an equality
$$\scl\left(\sum t_i g_i\right) = \sup_{\phi} \frac {\sum_i t_i \phi(g_i)} {2D(\phi)}$$
Because $Q/H^1$ is a Banach space, for any chain $\Gamma \in B_1^H(G)$ there exists a $\phi$
for which equality holds --- i.e.\/ for which $\scl(\Gamma)=\phi(\Gamma)/2D(\phi)$. Such a quasimorphism
is said to be {\em extremal} for $\Gamma$.

It is a fundamental problem, given $\Gamma$, to exhibit an explicit $\phi$ which is extremal for
$\Gamma$. There are essentially no examples of (hyperbolic) groups in which one knows how
to answer this problem for more than a handful of chains $\Gamma$. Upper bounds
on $\scl$ are obtained for (integral) chains $\Gamma$ by exhibiting $n\Gamma$ for some
$n$ as the oriented
boundary of a homotopy class of map $S \to K(G,1)$ for some compact oriented surface $S$ with
no disk or sphere components, and using the inequality $$\scl(\Gamma) = \inf_S \frac {-\chi(S)} {2n}$$
(see \cite{Calegari_rational} or \cite{Calegari_scl}, Prop.~2.10). A surface realizing 
$\scl(\Gamma)=-\chi(S)/2n$ is said to be {\em extremal} for $\Gamma$. For $\Gamma$ in $B_1^H(G)$ for
an arbitrary group $G$, an extremal surface need not exist. However, for a free group $F$,
it turns out that extremal surfaces always exist, and can be found by a polynomial time
algorithm (this is the Rationality Theorem from \cite{Calegari_rational}; the algorithm
is implemented by the program {\tt scallop} \cite{scallop}). For any surface $S$
and any homogeneous quasimorphism $\phi$ there is an inequality 
$-\chi(S)/2 \ge \scl(\partial S) \ge \phi(\partial S)/2D(\phi)$.
A surface $S$ and a quasimorphism $\phi$ certify each other as extremal (for $\partial S$)
if this inequality is an equality; i.e.\/ if $-\chi(S)/2 = \phi(\partial S)/2D(\phi)$.

In a free group, extremal (and other) surfaces bounding chains $n\Gamma$ are encoded 
combinatorially as {\em labeled fatgraphs}. The details of this labeling are 
explained in \S~\ref{fatgraph_subsection}, but the 
idea is just that the oriented edges of the fatgraph $Y$ are labeled by elements of $F$ in
such a way that changing the orientation inverts the label; and then the oriented boundary
of a surface thickening $S(Y)$ of the fatgraph determines a finite collection of cyclic words in $F$
which should represent $n\Gamma$ in $B_1^H(F)$.

Our second main result is that if we fix the topological type of a fatgraph $\hat{Y}$,
{\em most} labelings $Y$ give rise to extremal surfaces, and moreover we can explicitly
construct (from the combinatorics of $Y$) an extremal 
homogeneous quasimorphism $\overline{H}_Y$ which certifies that $S(Y)$ and $\overline{H}_Y$
are extremal:

\begin{fatgraph_thm}
For any combinatorial fatgraph $\hat{Y}$, if $Y$ is a random fatgraph over $F$ obtained by
labeling the edges of $\hat{Y}$ by words of length $n$, then $S(Y)$ is extremal 
for $\partial S(Y)$ and is certified by the extremal quasimorphism $\overline{H}_Y$, with
probability $1-O(C(\hat{Y},F)^{-n})$ for some constant $C(\hat{Y},F)>1$.
\end{fatgraph_thm}

This implies that for any integer $m$, {\em most} words $w$ in $F$ with 
$\cl(w)\le m$ satisfy $\cl(w)=m$ and $\scl(w)=m-1/2$.

To be useful in practice, it is important to have some idea of the size of the constants
$C(\hat{Y},F)$ arising in the Random Fatgraph Theorem. In \S~\ref{experimental_data_subsection} 
we tabulate the results of computer experiments for $F=F_2$ and for trivalent $\hat{Y}$; the
trivalent hypothesis significantly simplifies the construction of $\overline{H}_Y$ and the
verification of the certificate. The constants that arise are reassuringly small, affirming
the effectiveness of the Random Fatgraph Theorem.

\section{Injective endomorphisms of free groups are not always isometric}

\subsection{A question of Bardakov}

If $G$ is a group, and $G'$ is its commutator subgroup, the {\em commutator length} of
an element $g\in G'$ (denoted $\cl(g)$) is the least number of commutators in $G$ whose product is $g$.

\medskip

Bardakov \cite{Bardakov} asked the following question:

\begin{question}[Bardakov, \cite{Bardakov} qn.~2]
Let $\varphi:F \to F$ be an injective endomorphism of a nonabelian free group $F$. Does
$\cl(g) = \cl(\varphi(g))$ for all $g\in F'$?
\end{question}

The answer to Bardakov's question is {\em no}. We give two infinite families of examples
to substantiate this claim. The first family of examples use some facts from the theory of
$3$-manifold topology, and were inspired by a conversation with Geoff Mess.

\begin{example}[Complex of curves; \cite{Calegari_scl}, Ex.~4.44]\label{curve_complex_example}
Let $H$ be a handlebody of genus $3$. Let $\gamma$ be an essential simple closed
curve in $\partial H$, dividing $\partial H$ into two subsurfaces $S_1,S_2$ 
of genus $1$ and $2$ respectively. The inclusions $S_i \to \partial H$ are necessarily 
$\pi_1$-injective, though the inclusions $S_i \to H$ are typically not.
However, Dehn's lemma (see \cite{Hempel}) says that if $S_i \to H$ is {\em not} injective, there
is an essential simple closed curve $\gamma$ in $S_i$ that bounds an embedded disk in $H$.

The set of isotopy classes of essential simple closed curves in $\partial H$ are the vertices
of a graph $\C(\partial H)$ called the {\em complex of curves}. Two vertices are joined by an
edge in this complex if and only if they are represented by disjoint curves in $\partial H$.
If we declare that each edge has length $1$, the graph $\C(\partial H)$ becomes a (path) 
metric space, with distance function $d(\cdot,\cdot)$. Let $\C(H)$ denote the subset
of vertices consisting of essential simple closed curves in $\partial H$ that bound disks in $H$.

It is known (\cite{Hempel_pseudo}, Thm.~2.7) that there exist pseudo-Anosov 
mapping classes $\psi$ of $\partial H$ so that for any $\alpha \in \C(\partial H)$ the iterates $\psi^n(\alpha)$
satisfy $d(\psi^n(\alpha),\C(H)) \to \infty$. If $\beta$ is an arbitrary essential loop in $S_2$ 
then $d(\beta,\gamma)\le 1$, since $\beta$ and $\gamma=\partial S_2$ are disjoint. 
If $\psi$ is as above, and $n$ is such that
$d(\psi^n(\gamma),\C(H))\ge 2$, then $d(\psi^n(\beta),\C(H))\ge 1$ for all essential simple closed
curves $\beta$ in $S_2$. It follows from Dehn's lemma that the inclusion $\psi^n(S_2) \to H$ is 
$\pi_1$-injective. 

Let $F=\pi_1(S_2)$, a free group of rank $4$, and let $g\in F'$ be the conjugacy class associated
to the loop $\partial S_2$. A simple degree argument implies that $\cl(g)\ne 1$ and therefore $\cl(g)=2$.
Let $\varphi:F \to F$ be the endomorphism induced by the inclusion $S_2 \to \psi^n(S_2)\to H$ composed with
any injective homomorphism $\pi_1(H) \to F$. Since the image of $g$ in $\pi_1(H)$ is represented
by $\partial S_1$, this image is a commutator. Hence $\cl(\varphi(g))=1$.
\end{example}

\subsection{Stable commutator length}

If $G$ is a group, and $g\in G'$, the {\em stable commutator length} of $g$ (denoted $\scl(g)$)
is  the limit $\scl(g) := \lim_{n \to \infty} \cl(g^n)/n$. Stable commutator length is a more interesting
and subtle invariant than commutator length, and is connected to a broader range of mathematical
subjects, such as hyperbolic geometry, topology, symplectic dynamics, bounded cohomology, etc. See
\cite{Calegari_scl} for a systematic introduction.

It is convenient to extend the definition of (stable) commutator length 
to finite formal sums of elements. Suppose $g_i$ are a finite collection of elements in $G$ whose
product is in $G'$. Define $\cl(\sum g_i)$ to be the minimum of the commutator length of any product
$\prod_i g_i^{h_i}$ of conjugates of the $g_i$, and define $\scl(\sum g_i)$ to be the limit of $\cl(\sum g_i^n)/n$
as $n \to \infty$.

It is shown in \cite{Calegari_rational}, \S~2.4 (also see \cite{Calegari_scl}, \S~2.6) that $\scl$ extends
to a pseudo-norm on $B_1(G)$, the vector space of real group $1$-boundaries (in the sense of the bar
complex in group homology), and vanishes on the subspace $H$ spanned by chains of the
form $g^n-ng$ for $g\in G,n\in \Z$ and $g-hgh^{-1}$ for $g,h\in G$ (note that $H$ includes all torsion
elements).
Thus $\scl$ descends to a pseudo-norm on the quotient space $B_1^H:=B_1/H$.
When $G$ is a Gromov hyperbolic group (for example, when $G$ is free), $\scl$ defines a genuine {\em 
norm} on $B_1^H(G)$; this follows from \cite{Calegari_Fujiwara}, Thm.~A${}^\prime$ 
(the {\em separation theorem}).

If $G$ is a group in which (nontorsion) elements are not infinitely divisible, it is convenient to 
think of an element of $B_1^H$ as a (homologically trivial) finite formal real linear combination of
{\em primitive conjugacy classes}. Such objects arise frequently in low-dimensional geometry, 
e.g.\/ in the Selberg trace formula, or in Thurston's theory of train tracks.

\begin{definition}
A homomorphism between groups $\varphi:G \to H$ is {\em isometric} 
if $\scl_G(\Gamma) = \scl_H(\varphi(\Gamma))$ for all $\Gamma \in B_1^H(G)$. 
\end{definition}

Note that an isometric homomorphism between free groups is necessarily
injective.

\begin{example}
Any automorphism is isometric.
\end{example}

\begin{example}
An inclusion $G \to H$ that admits a section $H \to G$ is isometric.
\end{example}

\begin{example}
An endomorphism of a free group that sends every generator to a nontrivial power of itself
is isometric (\cite{Calegari_sails}, Cor.~3.16).
\end{example}

Our next family of examples depend on the main theorems of \cite{Calegari_faces}, and we refer
the reader to that paper for details.

\begin{example}[Nongeometric covers]\label{nongeometric_cover_example}
Let $F$ be a free group, and let $G$ be a finite index subgroup of $F$. Let $i:G \to F$ denote the
inclusion. A {\em realization} of a free
group is a conjugacy class of isomorphism $G \to \pi_1(\Sigma)$ where $\Sigma$ is a compact, connected,
oriented surface (necessarily with boundary). Associated to a realization there is a well-defined
chain $\partial \Sigma \in B_1^H(G)$. Say that a realization $G \to \pi_1(\Sigma)$ is {\em geometric}
if there is a realization $F \to \pi_1(S)$ and a finite cover $\Sigma \to S$ inducing 
$i:G \to F$. For a geometric realization, $i$ takes the equivalence class of the chain $\partial \Sigma$ to
the class of the chain $[F:G]\cdot \partial S$, and there are equalities:
$$-\chi(\Sigma)/2 = \scl_G(\partial \Sigma) = \scl_F(i_*\partial \Sigma) = [F:G]\cdot\scl_F(\partial S) = -[F:G]\cdot\chi(S)/2$$
However, if $G\to \pi_1(\Sigma)$ is nongeometric, it is {\em always} true that there is a {\em strict} inequality
$$\scl_G(\partial \Sigma) > \scl_F(i_*\partial \Sigma)$$
so that such $i_*$ are {\em never} isometric; see  
Proposition~\ref{prop:self_commensurating} below.

Note if the rank of $G$ is even,
there are many nongeometric realizations for which $\partial \Sigma$ is connected. This gives
many negative examples to Bardakov's question, since if $\scl(\varphi(g))<\scl(g)$ for some element
$g$ and some $\varphi$, then necessarily $\cl(\varphi(g^n))<\cl(g^n)$ for some $n$.
\end{example}

Note that every finite index subgroup $G$ of $F$ does in fact admit nongeometric realizations; hence
$G \to F$ is never isometric. Such an inclusion can be further composed with another
injective homomorphism to produce many examples.

\begin{definition}
A finitely generated subgroup $G < F$ is {\em self-commensurating} in $F$ if there is no
finitely generated subgroup $E < F$ with $G$ proper of finite index in $E$.
\end{definition}

We summarize this example in a proposition.

\begin{proposition}\label{prop:self_commensurating}
If $G \to F$ is an isometric homomorphism between finitely generated free groups, 
then the image of $G$ is self-commensurating in $F$.
\end{proposition}
The proof of this proposition is somewhat technical, depending on the main
results of \cite{Calegari_faces}. However, as the proposition is not used elsewhere
in the article, the reader who is not familiar with \cite{Calegari_faces} may skip it.
\begin{proof}
Let $G \to E$ be a proper inclusion of finite index between
finitely generated free groups. We show $G \to E$ is not isometric, and therefore
neither is $G \to F$.

Let $G \to \pi_1(\Sigma)$ be a nongeometric realization; i.e.\/ $\Sigma$ does not
cover a realization of $F$. It is easy to see that
every realization of a free group is extremal for its boundary; i.e.\/
$\Sigma$ is extremal for $\partial \Sigma$ in $G$, so $\scl_G(\partial \Sigma) = -\chi(\Sigma)/2$
(for the definition of an extremal surface, look ahead to \S~\ref{surface_subsection}).
Let $H$ be a subgroup of $G$ of finite index, normal in $E$. There is a realization
$H \to \pi_1(\til{\Sigma})$ for some finite cover $\til{\Sigma}$ of $\Sigma$,
and $\til{\Sigma}$ is extremal for $\partial \til{\Sigma}$ in $H$. Since
$H \to \pi_1(\til{\Sigma})$ is geometric with respect to $G$,
there is an equality $\scl_H(\partial \til{\Sigma}) = \scl_G(\partial \til{\Sigma}) = -\chi(\til{\Sigma})/2$.

On the other hand, since $G \to \pi_1(\Sigma)$ is not geometric with respect to $E$, 
neither is $H \to \pi_1(\til{\Sigma})$. Since $H$ is normal in $E$, 
there is some $e \in E$ which acts by conjugation on $H$ as an outer
automorphism $e_*$ of $H$ not in $\MCG(\til{\Sigma})$. By \cite{Calegari_faces} Thm.~A
the classes $\partial \til{\Sigma}$ and $e_*\partial \til{\Sigma}$ projectively
intersect the interiors of different top dimensional faces of the $\scl$ norm ball of $H$,
and therefore $\scl_H(\partial \til{\Sigma} + e_*\partial \til{\Sigma}) < 2 \partial \til{\Sigma}$.
Since $\scl$ is a norm, there is an inequality
$$\scl_E(\partial \til{\Sigma}) = \frac {\scl_H(\sum_{e \in E/H} e_* \partial \til{\Sigma})}
{[E:H]} < \scl_H(\partial \til{\Sigma}) = \scl_G(\partial \til{\Sigma})$$
(see \cite{Calegari_scl}, Cor.~2.81) and we are done.
\end{proof}

An interesting special case of Example~\ref{nongeometric_cover_example}
is to take $F=F_2$ and $G$ to be index $2$. Any realization 
$G \to \pi_1(\Sigma)$ has $\scl_G(\partial \Sigma)=1$,
and therefore any nongeometric realization produces an integral chain in $B_1^H(F)$ with $\scl<1$. 
Figure~\ref{histogram_1} is a histogram showing the distribution of $\scl_F(i_*\partial \Sigma)$
on 7500 ``random'' realizations of $G \to\pi_1(\Sigma)$ for a four-punctured sphere $\Sigma$.
\begin{figure}[htpb]
\labellist
\small\hair 2pt
\pinlabel $\frac{1}{2}$ at 20 5
\pinlabel $\frac{3}{4}$ at 120 5
\pinlabel $\frac{5}{6}$ at 153 5
\pinlabel $1$ at 220 5
\endlabellist
\centering
\includegraphics[scale=1]{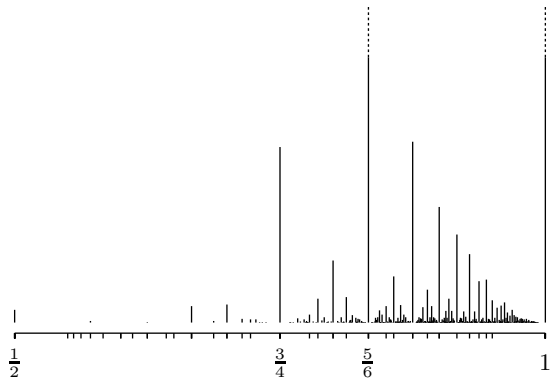}
\caption{Histogram showing distribution of $\scl(i_*\partial \Sigma)$ for 7500 realizations of $G$}\label{histogram_1}
\end{figure}

This figure suggests the following conjecture:

\begin{conjecture}[interval conjecture]
The set of values of $\scl$ on integral chains in $B_1^H(F_2)$ contains every rational number in the interval
$[3/4,1]$.
\end{conjecture}

In fact, it is {\em not} known whether the set of values of $\scl$ on integral chains
in any free group is dense in any interval, though it is known that this set is not discrete 
(see \cite{Calegari_sails}, Thm.~4.7).

\begin{example}
Let us look more closely at a single $2$-parameter family. Let $a,b,c$ generate an $F_3$, and let
$\varphi:F_3 \to F_2$ take $a \to a$, $b\to b^2$, $c \to b^{-1}ab$.
There is a $\Z^2$ in $\Aut(F_3)$ given by $(m,n):(a,b,c) \to (ab^m,b,b^nc)$. The image of the
chain $a+b+c+a^{-1}c^{-1}b^{-1}$ in $B_1^H(F_3)$ maps to $a+b^2+b^{-1}ab+a^{-1}b^{-1}a^{-1}b^{-1}$.
Precomposing with $(m,n)$ produces the chain $ab^{2m} + b^2 + b^{2n-1}ab + b^{-2m}a^{-1}b^{-1}a^{-1}b^{-1-2n}$.
Applying the automorphism $a \to aB, b \to b$ of $F_2$ to the image gives a chain which in $B_1^H(F_2)$
is equal to
$$w_{m,n}:= ab^{2m-1} + ab^{2n-1} + b^2 + a^{-2}b^{-2m-2n}$$
This is an example of a (2-parameter) {\em surgery family}, as defined in \cite{Calegari_sails}.
Computing $s(m,n):=\scl(w_{m,n})$ 
therefore reduces to the analysis of an explicit linear family of integer programming
problems. Such problems are in general beyond the reach of computer experiments, but this particular
family of examples is barely within reach of a rigorous analysis, and well within reach of a
heuristic analysis, implemented by the program {\tt sssf} \cite{Walker_sssf}.

\begin{center}
\begin{table}[htpb]
{\tiny
\begin{tabular}{ c c c c c c c c c c c c c }
  &     &      &       &       &        &         &       &         &         &         & 45/46 \\
  &     &      &       &       &        &         &       &         &         & 41/42   & 17/23 \\
  &     &      &       &       &        &         &       &         & 37/38   & 31/42   & 19/23 \\
  &     &      &       &       &        &         &       & 33/34   & 14/19   & 17/21   & 20/23 \\ 
  &     &      &       &       &        &         & 29/30 & 25/34   & 31/38   & 6/7     & 41/46 \\
  &     &      &       &       &        & 25/26   & 11/15 & 14/17   & 33/38   & 37/42   & 21/23 \\
  &     &      &       &       & 21/22  & 19/26   & 4/5   & 29/34   & 17/19   & 19/21   & 10/11  \\
  &     &      &       & 17/18 & 8/11   & 21/26   & 13/15 & 15/17   & 8/9     & 71/84   & 43/46 \\
  &     &      & 13/14 & 13/18 & 9/11   & 11/13   & 6/7   & 31/34   & 35/38   & 13/14   & 317/368 \\
  &     & 9/10 & 5/7   & 7/9   & 19/22  & 23/26   & 9/10  & 29/34   & 33/38   & 37/42   & 41/46 \\
  & 5/6 & 7/10 & 11/14 & 5/6   & 89/110 & 107/130 & 5/6   & 143/170 & 161/190 & 179/210 & 197/230 \\
1 & 3/4 & 4/5  & 6/7   & 8/9   & 10/11  & 12/13   & 14/15 & 16/17   & 18/19   & 20/21   & 22/23 \\
\\
\end{tabular}}
\caption{Values of $s(m,n)$ for $0\le n\le m\le 11$}\label{value_table}
\end{table}
\end{center}
Table~\ref{value_table} gives the value of $s(m,n)$ for $0\le n\le m \le 11$, and is included to
give the reader an indication of the variety of values of $\scl(\varphi_{m,n}(a+b+c+a^{-1}c^{-1}b^{-1}))$
possible in even a simple family of nonisometric injections $\varphi_{m,n}:F_3 \to F_2$.
\end{example}

Examples~\ref{curve_complex_example} and \ref{nongeometric_cover_example} show that it is quite easy
to construct injective homomorphisms between free groups that are not isometric. However, we will
show in \S~\ref{isometric_section} that a {\em random} 
homomorphism between free groups is isometric, and we further conjecture (and provide evidence
to suggest) that {\em every} injective
endomorphism of a free group of rank $2$ is isometric.

\section{Homomorphisms between free groups are usually isometric}\label{isometric_section}

In this section we describe a certain small cancellation condition guaranteeing that 
a homomorphism between free groups is isometric.
This condition is very similar to the condition $C'(1/12)$ studied in small cancellation theory
(see e.g.\/ \cite{Lyndon_Schupp}, Ch.~V), and is generic, in a sense to be made precise in the 
sequel. However, proving that this condition suffices to guarantee isometry depends on 
some technology developed in the papers \cite{Calegari_rational, Calegari_sails}, and
a careful inductive argument.

\subsection{Surfaces}\label{surface_subsection}

If $G$ is a group, let $X$ be a $K(G,1)$. Conjugacy classes in $G$ correspond to
free homotopy classes of loops in $X$.

Let $g_i \in G$ be a set of elements, and let $\Gamma:\coprod_i S^1_i \to X$ be a corresponding set of loops.
A map of a compact, oriented surface $f:S \to X$ is {\em admissible} for $\Gamma$ if there is a
commutative diagram
$$\begin{CD}
\partial S @>>> S \\
@V\partial f VV @Vf VV \\
\coprod_i S^1_i @>\Gamma >> \Sigma
\end{CD} $$
and an integer $n(S)$ for which $\partial f_*[\partial S] = n(S)[\coprod_i S^1_i]$
in $H_1$. The map is {\em monotone} if $\partial S \to \coprod_i S^1_i$ is homotopic to
an orientation-preserving cover (equivalently, if every component of $\partial S$ wraps with
positive degree around its image).

\begin{lemma}[\cite{Calegari_scl}, Prop.~2.74]\label{surface_lemma}
Let $g_1,\cdots,g_m$ be conjugacy classes in $G$, represented by $\Gamma:\coprod_i S^1_i \to X$.
Then
$$\scl(\sum_i g_i) = \inf_S \frac {-\chi^-(S)} {2n(S)}$$
where the infimum is taken over all surfaces $S$ and all maps $f:S \to X$ admissible for
$\Gamma$.
\end{lemma}

The notation $\chi^-(S)$ means the sum of Euler characteristics $\sum_i\chi(S_i)$ taken over those 
components $S_i$ of $S$ with $\chi(S_i)\le 0$. By \cite{Calegari_scl}, Prop.~2.13 it suffices to
restrict to {\em monotone} admissible surfaces. An admissible surface $S$ is 
{\em extremal} if equality is achieved.

\subsection{Fatgraphs}\label{fatgraph_subsection}

In free groups, most admissible surfaces --- and certainly all extremal ones --- can be
represented in an essentially combinatorial way, that is convenient for small cancellation
arguments. This combinatorial encoding is very similar to a method developed by Culler \cite{Culler}, 
though it is more or less equivalent to the theory of
{\em diagrams over surfaces} developed by Schupp \cite{Schupp}.

\begin{figure}[htpb]
\labellist
\small\hair 2pt
\pinlabel $b$ at 115 112
\pinlabel $b$ at 130 112
\pinlabel $A$ at 145 112
\pinlabel $B$ at 115 88
\pinlabel $B$ at 130 88
\pinlabel $a$ at 145 88
\pinlabel $a$ at 87 146
\pinlabel $b$ at 95 133
\pinlabel $A$ at 103 120
\pinlabel $A$ at 65 135
\pinlabel $B$ at 73 122
\pinlabel $a$ at 81 109
\pinlabel $a$ at 82 92
\pinlabel $B$ at 74 79
\pinlabel $B$ at 66 66
\pinlabel $A$ at 103 81
\pinlabel $b$ at 95 68
\pinlabel $b$ at 87 55
\endlabellist
\centering
\includegraphics[scale=1]{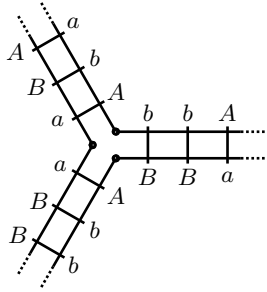}
\caption{Part of a thickened fatgraph over $F_2$ near a $3$-valent vertex}\label{junction}
\end{figure}

A {\em fatgraph} $Y$ is a graph in which each vertex has valence at least $3$,
together with a cyclic ordering of the edges incident at each vertex.
Such a graph can be thickened to a surface $S(Y)$ (or just $S$ is $Y$ 
is understood) in such a way that $Y$ embeds in $S(Y)$ as a
deformation retract (one also says $Y$ is a {\em spine} in $S(Y)$). 
A fatgraph $Y$ is {\em oriented} if $S(Y)$ is oriented. In the sequel we
assume all our fatgraphs are oriented. Note that $\chi(Y)=\chi(S(Y))$.

One can arrange for the deformation retraction $S(Y) \to Y$ to be locally injective
on $\partial S(Y)$. The preimages of the arcs of $Y$ give $\partial S(Y)$ a natural
cellular structure, in such a way that arcs of $\partial S(Y)$ map isomorphically to
arcs of $Y$, and vertices of $\partial S(Y)$ map to vertices of $Y$.
Two arcs of $\partial S$ mapping to the same edge of $Y$ are said to be {\em paired}.

A {\em fatgraph $Y$ over $F$} is an oriented fatgraph in which each arc
of $\partial S(Y)$ is labeled with a reduced, nontrivial
element of $F$ in such a way that paired arcs have labels which are inverse in $F$, 
and consecutive arcs (reading around $\partial S$) are reduced; 
see Figure~\ref{junction} for part of a fatgraph over $F_2$ near a
$3$-valent vertex (in this figure and elsewhere, we frequently
adopt the notation $A$ for $a^{-1}$ and so on). For such a fatgraph,
$\partial S$ is labeled by a finite collection of cyclically reduced cyclic words in $F$,
so we can (and do) think of the oriented boundary
$\partial S$ as an element of $B_1^H(F)$, which we denote $\partial S(Y)$. 

The basic fact we use is the following lemma, which is a restatement of \cite{Culler}, Thm.~1.4 in
the language of fatgraphs. Note that Culler proves his theorem only for surfaces with
connected boundary, but his argument generalizes with no extra work 
(an equivalent statement, valid for surfaces with disconnected boundary,
is also proved in \cite{Calegari_rational}, Lem.~3.4; also
see \cite{Calegari_scl} \S~4.3 for a discussion and references).

\begin{lemma}[Culler \cite{Culler}, Thm.~1.4 (fatgraph lemma)]\label{fatgraph_lemma}
Let $S$ be an admissible surface bounding a chain $\Gamma$. Then after possibly compressing
$S$ a finite number of times (thereby reducing $-\chi^-(S)$ without changing $\partial S$)
there is a fatgraph $Y$ over $F$ with $S(Y)=S$.
\end{lemma}

In the sequel $\hat{Y}$ will usually denote an abstract (unlabeled) fatgraph, and $Y$
will denote a labeled one.

\subsection{$\scl$ and word length}

In a free group $F$ with a fixed generating set, 
every element is represented by a unique reduced word, and every conjugacy class is
represented by a unique cyclically reduced cyclic word. 

Define $|\Gamma| = \min \sum |g_i|$, where $|\cdot|$ denotes word length in $F$,
and the minimum is taken over all representatives $\Gamma = \sum g_i$ of the class $\Gamma$ in $B_1^H$.
Note that if we take each $g_i$ to be primitive and cyclically reduced, and insist that no $g_i$
is conjugate to the inverse of some $g_j$ (in which case we could cancel $g_i$ and $g_j$), then
$|\Gamma| = \sum |g_i|$. In other words, any expression of $\Gamma$ as $\sum g_i$ either satisfies
$|\Gamma| = \sum |g_i|$, or can be reduced in an ``obvious'' way.

\begin{lemma}\label{upper_bound}
Let $\Gamma$ be an integral chain in $B_1^H(F)$. Then $\scl(\Gamma)\le |\Gamma|/2$.
\end{lemma}
\begin{proof}
In fact we prove the stronger statement that $\cl(\Gamma)\le |\Gamma|/2$. By the definition of
commutator length of a chain, it suffices to prove this in the case that $\Gamma$ is a single
word $g \in F'$. This means that every generator $x$ appears in $g$ as many times as $x^{-1}$
appears. Each such pair of letters can be canceled at the cost of a commutator, and the result
follows.
\end{proof}

The bound in Lemma~\ref{upper_bound} is not sharp. With more work, we obtain a sharp estimate.
The following lemma appeals at one point to a covering trick 
used in \cite{Calegari_sails}; since the trick is not used elsewhere in this paper, 
we refer the reader to \cite{Calegari_sails} for details.

\begin{lemma}\label{F2_upper_bound}
Let $\Gamma$ be an integral chain in $B_1^H(F_2)$. Then $\scl(\Gamma)\le |\Gamma|/8$.
\end{lemma}
\begin{proof}
Let $\Gamma = \sum g_i$ and by abuse of notation, suppose each $g_i$ is represented by a
cyclically reduced word. Suppose without loss of generality that there are at most
$|\Gamma|/2$ letters equal to one of $a$ or $A$. After applying the automorphism
$a \to ab,b\to b$ sufficiently many times, we obtain a new chain $\Gamma'=\sum h_i$ with at most
$|\Gamma|/2$ letters equal to one of $a$ or $A$, but with no $a^2$ or $A^2$ in any of
the cyclic words $h_i$.

Let $Y$ be a fatgraph with $\partial S(Y)=\Gamma'$, and let $S=S(Y)$. We
can decompose $S$ into a collection of at most $|\Gamma|/2$ rectangles pairing up $a$'s and
$A$'s, together with some subsurface $S'$ with at most $|\Gamma|$ corners, and edges
alternating between segments of $\partial S$ labeled by powers of $b$, and edges
corresponding to proper arcs in $S$.

Counting as in \cite{Calegari_sails}, each rectangle contributes $0$ to the ``orbifold Euler
characteristic'' of $S$, and each corner of $S'$ contributes $-1/4$. The total contribution
is therefore at most $|\Gamma|/4$, so $-\chi^-(S)\le |\Gamma|/4 - \chi(S')$. Now, it is possible
that $\chi(S')<0$, but since $S'$ has boundary components
labeled by elements of the abelian group $\langle b \rangle$, we can pass to a finite cover
of $S'$ and compress so that $\chi(S')$ can be made ``projectively'' as close to $0$ as
desired; this is explained in detail in \cite{Calegari_sails}, \S~3.3. Hence 
$\scl(\Gamma)=\scl(\Gamma')\le |\Gamma|/8$, as claimed.
\end{proof}

In fact, it is not much more work to extend this Lemma to free groups of arbitrary finite 
rank.  Let $F$ be freely generated by $x_1, \ldots, x_n$; if $\Gamma\in B_1^H(F)$, 
we denote by $|\Gamma|_i$ the number of times that $x_i$ and $x_i^{-1}$ appear in $\Gamma$.

\begin{proposition}\label{prop:Fn_upper_bound}
With notation as above, we have an inequality
\[
\scl(\Gamma) \le \frac{|\Gamma| - \max_i|\Gamma|_i}{4}
\]
for any $\Gamma \in B_1^H(F)$.
\end{proposition}
\begin{proof}
Without loss of generality, we may assume that $\max_i|\Gamma|_i = |\Gamma|_n$.  
As in the proof of Lemma~\ref{F2_upper_bound}, we may cut out all rectangles 
corresponding to matched pairs of $x_1$ and $x_1^{-1}$.  
What is left is an immersed subsurface $S'$ of $S$. An essential immersed subsurface of an extremal
surface is also extremal, by \cite{Calegari_faces}. Consequently $S'$
is extremal for its boundary $\Gamma'$, 
which lies in $B_1^H\left(\left\langle x_2, \ldots, x_n\right\rangle\right)$.  
We therefore have the inequality $\scl(\Gamma) \le \scl(\Gamma') + |\Gamma|_1/4$.  
Repeating this argument $n-1$ times yields 
\[
\scl(\Gamma) \le \scl(\Gamma'') + |\Gamma|_1/4 + \cdots + |\Gamma|_{n-1}/4
\]
where $\scl(\Gamma'') = 0$, since $\Gamma'' \in B_1^H\left(\left\langle x_n \right\rangle\right)$.
The proof follows.
\end{proof}

\begin{example}
The bound in Proposition~\ref{prop:Fn_upper_bound} is sharp, which we show by a family of examples.
We first recall the free product formula (\cite{Calegari_scl}, \S~2.7), 
which says that if $G_1$ and $G_2$ are arbitrary groups, and 
$g_i \in G_i'$ have infinite order, then 
$\scl_{G_1*G_2}(g_1g_2) = \scl_{G_1}(g_1) + \scl_{G_2}(g_2) + 1/2$.

Now, let $F$ be freely generated by $x_1,\ldots, x_n$ as above, and define
$$w_n = [x_1,x_2][x_3,x_4]\cdots [x_{n-1},x_n]$$ if $n$ is even, and
$$w_n = [x_1,x_2][x_3,x_4]\cdots [x_{n-4},x_{n-3}]x_{n-2}x_{n-1}x_nx_{n-1}^{-1}x_nx_{n-2}^{-1}x_n^{-2}$$
if $n$ is odd.

For each $i$, we have $\scl([x_i,x_{i+1}])=1/2$. Moreover, using {\tt scallop} (\cite{scallop}) one
can check that $\scl(x_{n-2}x_{n-1}x_nx_{n-1}^{-1}x_nx_{n-2}^{-1}x_n^{-2}) = 1$. The
free product formula then shows that $\scl(w_n)=(n-1)/2$, so
Proposition~\ref{prop:Fn_upper_bound} is sharp for all $n$.
\end{example}

\subsection{Small cancellation condition; first version}\label{small_cancellation_subsection}

A homomorphism between free groups is determined by the values of the generators,
which can be taken to be reduced words. In this section and the next, we 
define combinatorial conditions on these words which guarantee that
the homomorphism is an isometry of $\scl$.

For the sake of clarity, we first discuss a severe condition which makes the proof
of isometry easier. Then in \S~\ref{strong_condition_subsection} we discuss a weaker condition
which is generic (in a certain statistical sense, to be made precise) and which also implies
isometry, though with a slightly more complicated proof.

\begin{definition}\label{strong_cancellation_definition}
Let $A$ be a set, and let $F(A)$ be the free group on $A$. Let $U$ be a subset of $F(A)$ with
$U \cap U^{-1} = \emptyset$, and let
$\SS$ denote the set $U \cup U^{-1}$.
We say that $U$ satisfies condition (SA) if the following is true: 
\begin{itemize}
\item[(SA1)]{if $x,y \in \SS$ and $y$ is not equal to $x^{-1}$, then $xy$ is reduced; and}
\item[(SA2)]{if $x,y \in \SS$ and $y$ is not equal to $x$ or $x^{-1}$, then any common subword $s$
of $x$ and $y$ has length strictly less than $|x|/12$; and}
\item[(SA3)]{if $x \in \SS$ and a subword $s$ appears in at least two different 
positions in $x$ (possibly overlapping) then the length of $s$ is strictly less than $|x|/12$.}
\end{itemize}
Let $B$ be a set, and $\varphi:B \to U$ a bijection. Extend $\varphi$ to a homomorphism $\varphi:F(B) \to F(A)$.
We say $\varphi$ satisfies condition (SA) if $U$ satisfies condition (SA).
\end{definition}
Note that except for condition (SA1), this is the small cancellation condition $C'(1/12)$.
We will show the following:

\begin{proposition}\label{strong_cancellation_isometry}
Let $\varphi:F(B) \to F(A)$ be a homomorphism satisfying condition (SA). Then $\varphi$ is an
isometry of $\scl$.
\end{proposition}

Condition (SA1) for $\varphi$ means that if $g$ is a cyclically reduced word in $F(B)$, then the word
in $F(A)$ obtained by replacing each letter of $g$ by its image under $\varphi$ 
is also cyclically reduced. This condition is quite restrictive --- in particular it implies
that $|A|\ge |B|$, and even under these conditions it is
not ``generic'' --- but we will show how to dispense with it in \S~\ref{strong_condition_subsection}.
However, its inclusion simplifies the arguments in this section.

\begin{example}
The set $\lbrace aa,bb\rbrace$ satisfies (SA1). The set $\lbrace ab,ba\rbrace$ satisfies (SA1).
\end{example}

Suppose $\varphi:F(B) \to F(A)$ satisfies condition (SA), and let $Y$ be a fatgraph with 
$\partial S(Y)$ in the image of $\varphi$, i.e.\/ such that
$\partial S(Y)$ is a collection of cyclically reduced words of the form $\varphi(g)$. By
condition (SA1), each $\varphi(g)$ is obtained by concatenating words of the form $\varphi(x^{\pm})$ for
$x \in B$. We call these subwords {\em segments} of $\partial S(Y)$, as distinct from the
decomposition into {\em arcs} associated with the fatgraph structure.

\begin{definition}
A {\em perfect match} in $Y$ is a pair of segments $\varphi(x),\varphi(x^{-1})$ contained in
a pair of arcs of $\partial S(Y)$ that are matched by the pairing. A {\em partial match} in $Y$ is
a pair of segments $\varphi(x),\varphi(x^{-1})$ containing subsegments $s, s^{-1}$ in ``corresponding''
locations in $\varphi(x)$ and $\varphi(x^{-1})$ that are matched by the pairing.
\end{definition}

The existence of a perfect match will let us replace $Y$ with a ``simpler'' fatgraph. This is
the key to an inductive proof of Proposition~\ref{strong_cancellation_isometry}. The next lemma
shows how to modify a fatgraph $Y$ to promote a partial match to a perfect match.

\begin{lemma}\label{partial_to_perfect}
Suppose $Y$ contains a partial match. Then there is $Y'$ containing a perfect match with $S(Y')$
homotopic to $S(Y)$ and $\partial S(Y)=\partial S(Y')$.
\end{lemma}
\begin{proof}
The fatgraph $Y$ can be modified by a certain local move, illustrated in Figure~\ref{slide_junction}.

\begin{figure}[htpb]
\labellist
\small\hair 2pt
\pinlabel $x$ at 11 117
\pinlabel $X$ at -11 117
\pinlabel $x$ at -11 84
\pinlabel $X$ at 11 84
\pinlabel $X$ at 114 112
\pinlabel $x$ at 114 88
\pinlabel $x$ at 146 112
\pinlabel $X$ at 146 88
\pinlabel $\xrightarrow{slide}$ at 78 103
\endlabellist
\centering
\includegraphics[scale=1]{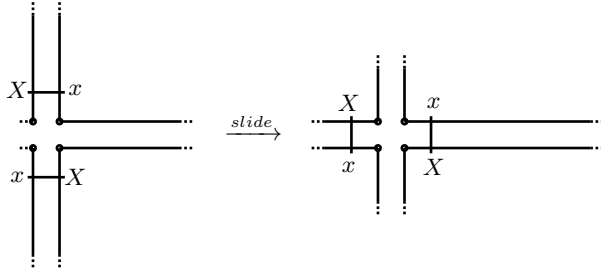}
\caption{A local move to replace a partial match with a perfect match}\label{slide_junction}
\end{figure}

This move increases the length of the paired subsegments by $1$.
Perform the move repeatedly to obtain a perfect match.
\end{proof}

\begin{remark}
The move illustrated in Figure~\ref{slide_junction} actually
occurs as the phenomenon of {\em branch migration}
in molecules of DNA, especially in certain $4$-valent junctions known as
Holliday junctions. See e.g.\/ \cite{panyutin_hsieh}.
\end{remark}

Each vertex $v$ of $Y$ of valence $|v|$ contributes $(|v|-2)/2$ to $-\chi(Y)$, in the sense that 
$-\chi(Y) = \sum_v (|v|-2)/2$. Since each vertex $v$ of $Y$ is in the image of $|v|$ vertices in
$\partial S$, we assign a weight of $(|v|-2)/2|v|$ to each vertex of $\partial S$. 

\begin{lemma}\label{perfect_match}
Let $Y$ be a fatgraph with $\partial S(Y)=\varphi(\Gamma)$ and suppose that $\varphi$ satisfies
(SA). Then either $Y$ contains a partial match, or $-\chi(Y)> |\Gamma|$.
\end{lemma}
\begin{proof}
Observe that $\partial S(Y)$ decomposes into $|\Gamma|$ segments, corresponding to the letters
of $\Gamma$. Suppose $Y$ does not contain a partial match. Then since each vertex contributes
$(|v|-2)/2|v|$ to $-\chi(Y)$, it suffices to show that each segment of
$\partial Y$ contains at least six vertices in its interior. 

Suppose not. Then some segment $\varphi(x)$
of $\partial Y$ contains a subsegment $s$ of length at least $|\varphi(x)|/6$ that does
not contain a vertex in its interior. Either $s$ contains a possibly smaller
subsegment $s'$ which is paired with some entire segment $\varphi(y)$, or at least 
half of $s$ is paired with some $s^{-1}$ in some $\varphi(y)$. In either case, since 
$s$ is not a partial match by hypothesis, we contradict either (SA2) or (SA3).

Thus each segment contributes at least $7\times ((3-2)/2\cdot 3) = 7/6$
to $-\chi(Y)$, and the lemma is proved.
\end{proof}

We now give the proof of Proposition~\ref{strong_cancellation_isometry}.

\begin{proof}
Suppose $\varphi:F(B) \to F(A)$ satisfies (SA) but is not isometric.

Let $Y$ be a fatgraph with $\partial S(Y)=\varphi(\Gamma)$ so that 
$\scl(\varphi(\Gamma)) \le -\chi(S(Y))/2 < \scl(\Gamma)$ (the existence of such a $Y$ follows from
\S~\ref{fatgraph_subsection}; for instance, we could take $Y$ to be
extremal). We will construct a new $Y'$ with $\partial S(Y') =\varphi(\Gamma')$ satisfying
$\scl(\varphi(\Gamma'))\le -\chi(S(Y'))/2 < \scl(\Gamma')$, and such that $Y'$ is shorter than $Y$.
By induction on the size of $Y$ we will obtain a contradiction.

By Lemma~\ref{upper_bound} and 
Lemma~\ref{perfect_match}, $Y$ contains a partial match, and by Lemma~\ref{partial_to_perfect}
we can modify $Y$ without affecting $\partial S(Y)$ or $\chi(Y)$ so that it contains a perfect match.
A perfect match cobounds a rectangle in $S=S(Y)$ that can be cut out, replacing $S$ with a ``simpler''
surface $S'$ for which $\partial S'$ is also in the image of $\varphi$. By Lemma~\ref{fatgraph_lemma},
there is some surface $S''$ with $-\chi(S'')\le -\chi(S')$ and $\partial S'' = \partial S'$, and
a fatgraph $Y'$ with $S(Y')=S''$. 

In the degenerate case that $S''$ is a disk, necessarily $S$
is an annulus, and both boundary components of $S$ consist entirely of perfect matches; hence
$\Gamma = g + g^{-1}$ and $\scl(\Gamma) = \scl(\varphi(\Gamma))=0$ in this case, contrary to hypothesis.
Otherwise $\partial S'' = \partial S' = \varphi(\Gamma')$ for some $\Gamma'$, and satisfies
$-\chi(S(Y')) \le -\chi(S') =-\chi(S(Y))-1$. 

On the other hand, $\Gamma$ can be obtained from $\Gamma'$ by gluing on a pair of pants; hence
$\scl(\Gamma) \le \scl(\Gamma')+1/2$. We have the following ``diagram of inequalities''\hfill
\begin{figure}[htpb]
\labellist
\small\hair 2pt
\pinlabel $\scl(\varphi(\Gamma))$ at 10 30
\pinlabel $\le$ at 50 30
\pinlabel $-\chi(S(Y))/2$ at 100 30
\pinlabel $<$ at 150 30
\pinlabel $\scl(\Gamma)$ at 190 30
\pinlabel $\scl(\varphi(\Gamma'))+1/2$ at 10 0
\pinlabel $\le$ at 50 0
\pinlabel $-\chi(S(Y'))/2+1/2$ at 100 0
\pinlabel $\scl(\Gamma')+1/2$ at 190 0
\pinlabel \rotatebox{90}{$\le$} at 100 15
\pinlabel \rotatebox{90}{$\ge$} at 190 15
\endlabellist
\centering
\includegraphics[scale=1]{inequalities}
\end{figure}
from which we deduce
$\scl(\varphi(\Gamma')) \le -\chi(S(Y'))/2 < \scl(\Gamma')$
as claimed. Since each reduction step reduces the length of $\partial S(Y)$, we obtain a contradiction.
\end{proof}

\subsection{Most homomorphisms between free groups are isometries}\label{strong_condition_subsection}

In this section we weaken condition (SA), allowing partial cancellation of adjacent words $\varphi(x)$
and $\varphi(y)$. Providing we quantify and control the amount of this cancellation, we obtain
a new condition (A) (defined below) which holds with high probability, and which implies isometry.

If two successive letters $x$, $y$ in a fatgraph
do not cancel, but some suffix of $\varphi(x)$ cancels some prefix
of $\varphi(y)$, we encode this pictorially by adding a {\em tag} to our fatgraph. A tag is an edge,
one vertex of which is $1$-valent. The two sides of the tag are then labeled by the maximal canceling
segments in $\varphi(x)$ and $\varphi(y)$. If $\Gamma$ is a chain, and $Y$ is a fatgraph with
$\partial S(Y)$ equal to the cyclically reduced representative of
$\varphi(\Gamma)$, then we can add tags to $Y$ to produce a fatgraph $Y'$ so that $\partial S(Y')$ is equal to the
(possibly unreduced) chain $\varphi(\Gamma)$. 

\begin{definition}
Let $A$ be a set, and let $F(A)$ be the free group on $A$. Let $U$ be a subset of $F(A)$ with 
$U \cap U^{-1} = \emptyset$, and let $\SS$ denote the set $U\cup U^{-1}$. We say that $U$
satisfies condition (A) if there is some non-negative real number $T$ such that the following is true:
\begin{itemize}
\item[(A1)]{the maximal length of a tag is $T$; and}
\item[(A2)]{if $x,y\in \SS$ and $y$ is not equal to $x$ or $x^{-1}$, then any common subword
$s$ of $x$ and $y$ has length strictly less than $(|x|-2T)/12$; and}
\item[(A3)]{if $x\in \SS$ and a subword $s$ appears in at least two different positions in $x$
(possibly overlapping) then the length of $s$ is strictly less than $(|x|-2T)/12$.}
\end{itemize}
Let $B$ be a set, and $\varphi:B \to U$ a bijection. Extend $\varphi$ to a homomorphism 
$\varphi:F(B) \to F(A)$. We say $\varphi$ satisfies condition (A) if $U$ satisfies condition (A).
\end{definition}

Notice that condition (SA) is a special case of condition (A) when $T=0$.

\begin{proposition}\label{condition_A_isometry}
Let $\varphi:F(B) \to F(A)$ be an homomorphism between free groups
satisfying condition (A). Then $\varphi$ is an isometry
of $\scl$. That is, $\scl(\Gamma) = \scl(\varphi(\Gamma))$ for all chains $\Gamma \in B_1^H(F(B))$.
In particular, $\scl(g) = \scl(\varphi(g))$ for all $g\in F(B)'$.
\end{proposition}
\begin{proof}
The proof is essentially the same as that of Proposition~\ref{strong_cancellation_isometry}, 
except that we need to be slightly more careful computing $\chi(Y)$. We call the edges in a tag {\em ghost edges},
and define the valence of a vertex $v$ to be the number of non-ghost edges incident to it. Then
$-\chi(Y) = \sum_v (|v|-2)/2$ where the sum is taken over all ``interior'' vertices $v$ --- i.e.\/ those
which are not the endpoint of a tag. 

The proof of Lemma~\ref{perfect_match} goes through exactly as before, showing that
either $Y$ contains a partial match, or $-\chi(Y)> |\Gamma|$. To see this, simply repeat the
proof of Lemma~\ref{perfect_match} applied to $Y$ with the tags ``cut off''. Partial matches
can be improved to perfect matches as in Lemma~\ref{partial_to_perfect}. Note that this move
might unfold a tag. 

If $Y$ is a fatgraph with $\partial S(Y)=\varphi(\Gamma)$ and
$\scl(\varphi(\Gamma)) \le -\chi(S(Y))/2 < \scl(\Gamma)$, we can find a perfect match and cut
out a rectangle, and the induction argument proceeds exactly as in the proof of
Proposition~\ref{strong_cancellation_isometry}.
\end{proof}

Fix $k,l$ integers $\ge 2$. We now explain the sense in which a random homomorphism from
$F_k$ to $F_l$ will satisfy condition (A). Fix an integer $n$, and let $F_l(\le n)$ denote
the set of reduced words in $F_l$ (in a fixed free generating set) of length at most $n$.
Define a {\em random homomorphism of length $\le n$} to be the homomorphism
$\varphi:F_k \to F_l$ sending a (fixed) free generating set for $F_k$ to $k$ randomly
chosen elements of $F_l(\le n)$ (with the uniform distribution).

\begin{theorem}[Random Isometry Theorem]\label{isometry_theorem}
A random homomorphism $\varphi:F_k \to F_l$ of length $n$
between free groups of ranks $k,l$ is an isometry of $\scl$ with probability $1-O(C(k,l)^{-n})$
for some constant $C(k,l)>1$.
\end{theorem}
\begin{proof}
By Proposition~\ref{condition_A_isometry} it suffices to show that a random homomorphism
satisfies condition (A) with sufficiently high probability.

Let $u_1,\cdots, u_k$ be the images of a fixed free generating set for $F_k$, thought
of as random reduced words of length $\le n$ in a fixed free generating set and their
inverses for $F_l$. First of all, for any $\epsilon>0$, we can assume
with probability at least $1-O(C^{-n})$ for some $C$ that the length of every $u_i$ is between
$n$ and $(1-\epsilon)n$. Secondly, the number of reduced words of length $\epsilon n$
is (approximately) $(2l-1)^{\epsilon n}$, so the chance that the maximal length of
a tag is more than $\epsilon n$ is at least $1-O(C^{-n})$. So we restrict attention
to the $\varphi$ for which both of these condition hold.

If (A2) fails, there are indices $i$ and $j$ and a subword $s$ of $u_i$ of length
at least $n(1-3\epsilon)/12 \ge n/13$ (for large $n$) so that either $s$ or $s^{-1}$ is a subword of $u_j$.
The copies of $s^\pm$ are located at one of at most $n$ different places in $u_i$
and in $u_j$; the chance of such a match at one specific location is approximately $(2l-1)^{-n/13}$,
so the chance that (A2) fails is at most $k^2n^2(2l-1)^{-n/13} = O(C^{-n})$ for suitable $C$.

Finally, if (A3) fails, there is an index $i$ and a subword $s$ of $u_i$ of length
at least $n/13$ that appears in at least two different locations. It is possible that
$s$ overlaps itself, but in any case there is a subword of length at least $|s|/3$ that
is disjoint from some translate. If we examine two specific disjoint subsegments of length
$n/39$, the chance that they match is approximately $(2l-1)^{-n/39}$. Hence the chance that (A3)
fails is at most $kn^2(2l-1)^{-n/39} = O(C^{-n})$ for suitable $C$. Evidently $C$ depends only
on $k$ and $l$. The lemma follows.
\end{proof}

\begin{corollary}\label{isometry_exists}
Let $k,l \ge 2$ be integers. There are (many) isometric homomorphisms $\varphi:F_k \to F_l$.
\end{corollary}

\begin{lemma}\label{addition_lemma}
Let $F$ be a finitely generated free group. The following holds:
\begin{enumerate}
\item{if there are integral chains $\Gamma_1,\Gamma_2$ in $B_1^H(F)$
such that $\scl(\Gamma_i) = t_i$, then there is an integral chain $\Gamma$ in $B_1^H(F)$ with
$\scl(\Gamma)=t_1+t_2$; and}
\item{if there are elements $g_1,g_2$ in $F'$ such
that $\scl(g_i)=t_i$, then there is an element $g \in F'$ with $\scl(g)=t_1+t_2+1/2$.}
\end{enumerate}
\end{lemma}
\begin{proof}
Let $F_1,F_2$ be copies of $F$, and let $\sigma_i:F \to F_i$ be an isomorphism. Then in case (1)
the chain $\sigma_1(\Gamma_1) + \sigma_2(\Gamma_2)$ in $F_1*F_2$ has $\scl$ equal to $t_1+t_2$, and
in case (2) the element $\sigma_1(g_1)\sigma_2(g_2)$ has $\scl$ equal to $t_1+t_2+1/2$; see
\cite{Calegari_scl}, \S~2.7. Now choose an isometric homomorphism from $F_1*F_2$ to $F$, 
which exists by Corollary~\ref{isometry_exists}.
\end{proof}

\begin{corollary}
Let $F$ be a countable nonabelian free group. The image of $F'$ under $\scl$ contains elements
congruent to every element of $\Q$ mod $\Z$. Moreover, the image of $F'$ under $\scl$
contains a well-ordered sequence of values with ordinal type $\omega^\omega$.
\end{corollary}
\begin{proof}
These facts follow from Lemma~\ref{addition_lemma} plus the Denominator Theorem and Limit Theorem
from \cite{Calegari_sails}.
\end{proof}

\section{Isometry conjecture}\label{isometry_conjecture_section}

\begin{conjecture}[Isometry conjecture]\label{isometry_conjecture}
Let $\varphi:F_2 \to F$ be any injective homomorphism from a free group of rank $2$ to a free group $F$.
Then $\varphi$ is isometric. 
\end{conjecture}

\begin{remark}
Since free groups are Hopfian by Malcev \cite{Malcev}, any homomorphism from $F_2$ to a 
free group $F$ is either injective, or factors through a cyclic group. Furthermore, since $F_2$
is not proper of finite index in any other free group, every $F_2$ in $F$ is self-commensurating,
and therefore no counterexample to the conjecture can be constructed by the method of 
Proposition~\ref{prop:self_commensurating}.

Since any free group admits an injective homomorphism into $F_2$, and since $\scl$ is monotone nonincreasing
under any homomorphism between groups, to prove Conjecture~\ref{isometry_conjecture} it suffices
to prove it for endomorphisms $\varphi:F_2 \to F_2$.
\end{remark}

\begin{remark}
In view of Example~\ref{nongeometric_cover_example}, rank 2 cannot be replaced with rank 3 in 
Conjecture~\ref{isometry_conjecture}.
\end{remark}

Conjecture~\ref{isometry_conjecture} has been tested experimentally on all
cyclically reduced homologically trivial words of length 11 in $F_2$, and all endomorphisms $F_2 \to F_2$
sending $a \to a$ and $b$ to a word of length $4$ or $5$. It has also been tested on thousands of
``random'' longer words and homomorphisms. The experiments were carried
out with the program {\tt scallop} (\cite{scallop}), which implements the
algorithm described in \cite{Calegari_rational} and \cite{Calegari_scl}.

\medskip

In order to give some additional evidence for the conjecture beyond the results
of \S~\ref{strong_condition_subsection}, we prove it in a very specific (but interesting) case for which
the small cancellation conditions (SA) and (A) do not hold.

\begin{proposition}\label{special_case_isometry}
The homomorphism $\varphi:F_2 \to F_2$ defined on generators $a,b$ by
$\varphi(a) = abA$, $\varphi(b)=b$ is an isometry.
\end{proposition}
\begin{proof}
The proof is by induction, following the general strategy of the
proof of Proposition~\ref{strong_cancellation_isometry} and
Proposition~\ref{condition_A_isometry}, but with a more complicated combinatorial
argument. As in the proof of those propositions, we assume to the contrary that
there is some $\Gamma$ and a fatgraph $Y$ with $\partial S(Y)=\varphi(\Gamma)$ so
that $\scl(\varphi(\Gamma)) \le -\chi(S(Y))/2 < \scl(\Gamma)$. If we can find
a partial match in $Y$, then we can cut out a rectangle and get a simpler fatgraph
$Y'$ and a chain $\Gamma'$ so that $\scl(\varphi(\Gamma'))\le -\chi(S(Y'))/2 < \scl(\Gamma')$,
and we will be done by induction. We show now that such a partial match must exist.

Note that each consecutive string $a^m$ in $\Gamma$ gives rise to a string of
the form $ab^mA$ in $\Gamma'$, and each $b^m$ in $\Gamma$ gives rise to a string
of the form $b^m$. We call copies of $b$ or $B$ in $\varphi(\Gamma)$
of the first kind {\em fake}, and copies of $b$ or $B$ in $\varphi(\Gamma)$ of the
second kind {\em real}. Every $b$ (real or fake) must pair with some $B$ (real or fake)
in $Y$. If a real $b$ pairs with a real $B$, or a fake $b$ with a fake $B$, then we
obtain a partial match, which can be improved to a perfect match by Lemma~\ref{partial_to_perfect},
and then cut out, completing the induction step.

So we assume to the contrary that there are no partial matches, and every real $b$
pairs a fake $B$ and conversely. Assume for the moment that $\Gamma$ has no subwords
that are powers of the generators (these are called {\em abelian loops} in \cite{Calegari_sails},
and we use this terminology in what follows). Then each string of real $b$'s or $B$'s 
in $\varphi(\Gamma)$ is followed by $a$ and preceded by $A$, whereas each string of
fake $b$'s or $B$'s in $\varphi(\Gamma)$ is followed by $A$ and preceded by $a$.
Moreover, each $a$ is followed by a fake $b$ or $B$, and preceded by a real $b$ or $B$.
These facts together imply that each $a$ or $A$ in $\varphi(\Gamma)$ 
is contained in an edge of length {\em exactly $1$}. 
From this we obtain a lower bound on $-\chi(S(Y))$, as
follows. If $\Gamma$ contains $n$ segments of the form $a^m$ and $n$ of the form
$b^m$, then (assuming there are no abelian loops), there are exactly $n$ edges
of $Y$ which pair a single $a$ with an $A$. Removing these edges leaves a fatgraph with
no $1$-valent edges, since $a$ edges are never adjacent at a vertex. Hence
each such edge contributes at least $1$ to $-\chi$, and we obtain the
inequality $-\chi(S(Y))/2\ge n/2$.

However, we claim that the form of $\Gamma$ implies that $\scl(\Gamma) \le n/2$,
contrary to hypothesis. This shows that there is a partial match after all, and
therefore $Y$ can be simplified. But by induction this shows that no such $\Gamma$
and $Y$ can exist, and the proposition will be proved.

The inequality $\scl(\Gamma) \le n/2$ follows easily from 
the method of \cite{Calegari_sails} (in fact, the stronger inequality 
$\scl(\Gamma)\le (n-1)/2$ (achieved for $\Gamma = abAB$)
is true, but we do not need this). In \S~3 of that paper, it is shown that
for $\Gamma$ of the desired form, $\scl(\Gamma) = \min_{y\in Y} n/2 - (\kappa_A(y) + \kappa_B(y))/2$,
where $\kappa_A$ and $\kappa_B$ are certain piecewise linear {\em non-negative}
functions, and $y$ ranges over a certain rational convex polyhedron $Y$. The
desired inequality (and the proof) follows, ignoring abelian loops.

Each abelian loop of $\Gamma$ reduces the count of $a$ edges in $Y$ by $1$, but 
(\cite{Calegari_scl}, p.~9) also reduces the upper bound on $\scl(\Gamma)$ by
$1/2$. In other words:
$$\scl(\Gamma) = \min_{y \in Y} n/2 - \#\lbrace\text{abelian loops}\rbrace/2 - (\kappa_A(y) + \kappa_B(y))/2$$
so the desired inequality holds in this case too.
\end{proof}

\begin{example}
In \cite{Calegari_sails} \S~4.1 it is shown that $\scl(a^m + B^m + aBA^{m+1}b^{m+1}) = (2m-1)/2m$ for
$m\ge 2$. Under $\varphi$, the image of $a^m$ and $B^m$ cancel, and one obtains the identity
$\scl([a,b][a,B^{m+1}])=(2m-1)/2m$ for $m\ge 2$. This family of words is discussed in
\cite{Calegari_scl} \S~4.3.5 and an explicit collection of bounding surfaces exhibited. 
Proposition~\ref{special_case_isometry} certifies these surfaces as extremal.
\end{example}

\begin{example}
The homomorphism $\varphi$ arises naturally as the inclusion of $F_2$ as a factor in $F_\infty$,
the first term in a short exact sequence $F_\infty \to F_2 \to \Z$, where the $F_2 \to \Z$ kills
one of the generators. It is not true that inclusions of bigger factors $F_n$ in $F_\infty$ are
isometric. For example, $\scl([a,b][c,d])=3/2$, but $\scl([a,a^b][a^{b^2},a^{b^3}])=1$.
\end{example}

\section{Labelings of a fatgraph are usually extremal}

In this section, we show that for an arbitrary topological fatgraph $\hat{Y}$,
a random labeling of its edges by words of length $n$ is extremal for
its boundary with probability $1-C^{-n}$. Notice that such a labeling defines a random {\em groupoid}
homomorphism from the edge groupoid of $\hat{Y}$ to a free group $F$. Such a groupoid homomorphism
in turn induces a homomorphism from $\pi_1(\hat{Y})$ to $F$, but such a homomorphism will {\em never}
satisfy property (A) if $\hat{Y}$ has more than one vertex, since the generators of $\pi_1(\hat{Y})$
necessarily map to words in $F$ with big overlaps, corresponding to common subedges of $\hat{Y}$.

One significant feature of our construction is that the proof that a typical labeling $Y$ of
$\hat{Y}$ is extremal comes together with a {\em certificate}, in the form of
a (dual) extremal quasimorphism. Producing explicit extremal quasimorphisms for given 
elements is a fundamental, but very difficult problem, and as far as we know this is the first
example of such a construction for ``generic'' elements (in any sense) in a hyperbolic group.

The construction of the extremal quasimorphism dual to a ``generic'' fatgraph is somewhat
involved; however, there is a special case where the construction is extremely simple, namely
that of trivalent fatgraphs. Therefore we first present the construction and the proofs
in the case of trivalent fatgraphs, deferring a discussion of more general fatgraphs 
to \S~\ref{higher_valence}.

\subsection{Quasimorphisms}

Recall that if $G$ is a group, a {\em quasimorphism} is a function $\phi:G \to \R$ for which
there is a least non-negative number $D(\phi)$ (called the {\em defect}) so that for all $g,h \in G$
there is an inequality
$$|\phi(gh)-\phi(g)-\phi(h)| \le D(\phi)$$
A quasimorphism is further said to be {\em homogeneous} if it satisfies $\phi(g^n)=n\phi(g)$ for
all $g \in G$ and all integers $n$. 

If $\phi$ is an arbitrary quasimorphism, its {\em homogenization} $\overline{\phi}$ is defined to
be the limit $\overline{\phi}(g):=\lim_{n \to \infty} \phi(g^n)/n$. It is a fact that $\overline{\phi}$
with this definition is a homogeneous quasimorphism, with $D(\overline{\phi})\le 2D(\phi)$. See
\cite{Calegari_scl}, \S~2.2.

Rhemtulla \cite{Rhemtulla}, and then later Brooks \cite{Brooks}, gave an elementary 
construction of quasimorphisms on free groups, which we refer to as \emph{counting quasimorphisms}. 
For a word $w \in F$, define the {\em big counting function} $C_w$ by the formula
$$C_w(v) = \text{number of copies of }w\text{ in }v$$
Then $H_w = C_w - C_{w^{-1}}$ is a quasimorphism, called the {\em big 
counting quasimorphism} for $w$.  The function $H_w$ counts the difference between
the number of copies of $w$ and of $w^{-1}$ in a given word, and its homogenization $H_w$ counts
the difference between the number of copies in the associated {\em cyclic} word. For such
functions one has $D(\overline{H}_w)=2D(H_w)$.

Following Epstein--Fujiwara \cite{Epstein_Fujiwara}, we define a variant on this construction 
as follows. For a given set $S \subseteq F$, denote by 
$S^{-1}$ the set of inverses of elements of $S$, and define the {\em small counting function} $c_S$ by
$$c_S(v) = \text{maximal number of disjoint copies of elements of }S\text{ in }v$$
So for example, $c_{\lbrace ab,\,ba,\,bb \rbrace}(abba) = 2$.  
Define $h_S: = c_S - c_{S^{-1}}$ to be the {\em small 
counting quasimorphism} for $S$. 

A significant property of small counting quasimorphisms (by contrast with the big
counting quasimorphisms) is that there is a {\em universal} bound on their defect,
which (except in rare cases) is sharp.

\begin{lemma}
\label{lem:small_defect}
For any $S \subseteq F$, we have $D(h_S) \le 3$ and $D(\overline{h}_S)\le 6$.
\end{lemma}
\begin{proof}
There is a standard method to estimate defect of counting quasimorphisms and their variants, 
which we describe. First, note that $h_S$ is {\em antisymmetric}, i.e.\/ $h_S(w) = -h_S(w^{-1})$ 
for all $w\in F$. This is a property that will be shared by all the quasimorphisms we consider in
the sequel. 

Now, given any $g,h$ there are reduced words $k,l,m$ so that the words $kL,lM,mK$ are all
reduced, and represent $gh,g^{-1},h^{-1}$ respectively. We think of the words $k,l,m$
as the labels on the incoming edges on a tripod $Y$ (thought of as an especially simple kind of fatgraph)
and observe that the oriented boundary $\partial S(Y)=kL+lM+mK$. Since
$h_S$ is antisymmetric, it suffices to compute $h_S(kL+lM+mK)$.

We refer to the $3$-valent vertex of the tripod as the {\em junction}. By the definition of
small counting functions, if $k$, $L$ and $kL$ are all reduced words, then $0 \le c_S(kL-k-L)\le 1$,
since any collection of disjoint $S$-words in $k$ and $L$ produces such a collection in $kL$ 
not crossing the junction, whereas any collection of disjoint $S$-words in $kL$ contains at most one
that crosses the junction. Symmetrizing, $|h_S(kL-k-L)|\le 1$.
But then we can compute
$$|h_S(kL+lM+mK)| = |h_S(kL-k-L) + h_S(lM-l-M) + h_S(mK-m-K)| \le 3$$
Homogenizing multiplies the defect by at most $2$, and the lemma is proved.
\end{proof}

It is the sharpness of this estimate that will allow us to use small counting quasimorphisms
to calculate $\scl$ {\em exactly}.

\subsection{Labeling fatgraphs}

We use the notation $\hat{Y}$ for an abstract (unlabeled) fatgraph, and $Y$ for a labeling
of $\hat{Y}$ by words in $F$; i.e.\/ a fatgraph over $F$ (see \S~\ref{fatgraph_subsection}).
A labeling {\em of length $n$} is a reduced labeling for which every edge of $Y$ is a word of
length $n$. 

By our convention, boundary words in $\partial S(Y)$ must be cyclically reduced. For a
labeling in which boundary words are not reduced, one can ``fold'' adjacent canceling letters
to produce tags as in \S~\ref{strong_condition_subsection}. One can then either cut off
tags, or think of them as ``ghost'' edges to be ignored. Note that folding in this sense is a
restricted kind of folding in the sense of Stallings \cite{Stallings}, since the folding must
respect the cyclic ordering of edges incident to a vertex. Hence a fatgraph which is completely
folded (equivalently, for which $\partial S(Y)$ is cyclically reduced) is not a priori
$\pi_1$-injective.

\subsection{The vertex quasimorphism construction}

In this section, we construct a (counting) quasimorphism on $F$ from a fatgraph 
$Y$ over $F$.  We will call this the vertex quasimorphism of $Y$.  We will see 
that this vertex quasimorphism is typically extremal for $\partial S(Y)$.

Define a set $\sigma_Y$ on a labeled fatgraph $Y$ over $F$ as follows: every 
boundary component of $S(Y)$ decomposes into a union of arcs, and each arc is 
labeled by an element of $F$.  Between each pair of arcs is a vertex of $\partial S(Y)$ 
(associated to a vertex of $Y$).  For each vertex of $Y$ and each pair of incident 
arcs with labels $u$ and $v$ ($u$ comes into the vertex; $v$ leaves it), decompose 
$u$ and $v$ into $u = u_1u_2$, $v = v_1v_2$, where usually we expect $u_1$ and $u_2$ to each be
approximately half the length of $u$, and similarly for $v_1$, $v_2$, $v$, and add the 
word $u_2v_1$ to the set $\sigma_Y$.  There is some flexibility here in the 
phrase ``about half the length'' which will not affect our later arguments; 
in fact this flexibility indicates possible other constructions, in which the 
pieces have different sizes, bounded length, etc.

A {\em vertex quasimorphism} for $Y$ is a small counting quasimorphism of the form $h_{\sigma_Y}$.  
See Figure \ref{fig:vert_quasi} for an example. In this figure, $\sigma_Y$ is the set
$$\sigma_Y = \{bbAb,\, aBAA,\, aaaa,\, AbAA,\, AbaB,\, BBaB\}$$ 
Note that we have not broken the edges exactly in half, or even in the same place on either side.

\begin{figure}[htpb]
\labellist
\small\hair 2pt
%bottom outside
\pinlabel $a$ at 27 66
\pinlabel $B$ at 65 25
\pinlabel $B$ at 127 25
\pinlabel $a$ at 166 66
%bottom inside
\pinlabel $A$ at 46 72
\pinlabel $b$ at 73 44
\pinlabel $b$ at 117 44
\pinlabel $A$ at 145 72
%middle bottom
\pinlabel $b$ at 51 86
\pinlabel $A$ at 80 86
\pinlabel $b$ at 111 86
\pinlabel $A$ at 140 86
%middle top
\pinlabel $B$ at 51 107
\pinlabel $a$ at 80 107
\pinlabel $B$ at 111 107
\pinlabel $a$ at 140 107
%top inside
\pinlabel $A$ at 46 121
\pinlabel $A$ at 73 149
\pinlabel $A$ at 117 150
\pinlabel $b$ at 145 121
%top outside
\pinlabel $a$ at 27 130
\pinlabel $a$ at 65 167
\pinlabel $a$ at 127 167
\pinlabel $B$ at 166 124
\endlabellist
\centering
\includegraphics[scale=1]{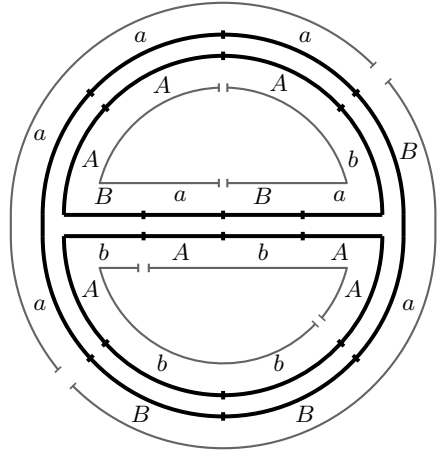}
\caption{The vertex quasimorphism construction on a thrice-punctured sphere.}\label{fig:vert_quasi}
\end{figure}

\begin{lemma}
\label{lem:value_of_vert_quasi}
If no element of $\sigma_Y^{-1}$ appears in the boundary $\partial S(Y)$, then there 
is an inequality $\overline{h}_{\sigma_Y}(\partial S(Y)) \ge \sum_v |v|$, where the sum is taken over all
vertices $v$, and $|v|$ is the valence of the vertex $v$.
\end{lemma}
\begin{proof}
Note that since the components of $\partial S(Y)$ are cyclic words (rather than words), it only
makes sense to apply the homogenized functions $\overline{c}$ and $\overline{h}$ to them.

Since no element of $\sigma_Y^{-1}$ appears in $\partial S(Y)$, we have 
$\overline{c}_{\sigma_Y^{-1}}(\partial S(Y))=0$, so 
$\overline{h}_{\sigma_Y}(\partial S(Y)) = \overline{c}_{\sigma_Y}(\partial S(Y))$.  
For every vertex of $Y$ and for each incident 
edge, we have a word in $\sigma_Y$.  By construction, these words do not overlap in 
the boundary chain $\partial S(Y)$, so the value of $\overline{c}_{\sigma_Y}(\partial S(Y))$ is 
at least as big as $\sum_v |v|$.  
\end{proof}

\begin{remark}
Note that it is possible for a strict inequality in Lemma~\ref{lem:value_of_vert_quasi}, since there
may be many different ways to put disjoint copies of elements of $\sigma_Y$ in $\partial S(Y)$.
However, if $Y$ is trivalent and $\sigma_Y$ satisfies the hypotheses of the lemma, then there is
an {\em equality} $\overline{h}_{\sigma_Y}(\partial S(Y))$ is equal to three times the number of vertices of $Y$.
\end{remark}

\subsection{Trivalent fatgraphs are usually extremal}\label{extremal_inclusions_section}

We say that a fatgraph $Y$ over $F$ satisfies condition (SB) if there is a choice of
$\sigma_Y$ as above so that no element of $\sigma_{Y}^{-1}$ appears in $\partial S(Y)$.

\begin{lemma}
\label{lem:trivalent_are_extremal}
If a trivalent labeled fatgraph $Y$ satisfies condition (SB), then both $S(Y)$ and 
$\overline{h}_{\sigma_{Y}}$ are extremal for the boundary $\partial S(Y)$, and certify each other.
\end{lemma}
\begin{proof}

For a trivalent graph, $\bar{h}_{\sigma_{Y}}(\partial S(Y)) \ge 3V$, 
where $V$ is the number of vertices, by Lemma \ref{lem:value_of_vert_quasi}. By Bavard duality, and
Lemma~\ref{lem:small_defect} there is an inequality
$$\scl(\partial S(Y)) \ge \frac{3V}{2D(\bar{h}_{\sigma_Y})} \ge \frac{3V}{4D(h_{\sigma_Y})} \ge \frac {V}{4}$$
On the other hand, since $Y$ is trivalent, the number of edges is $3V/2$, so $\chi(S(Y))=-V/2$.
Hence we get a chain of inequalities
$$\scl(\partial S(Y))  \ge \frac{3V}{2D(\bar{h}_{\sigma_{Y}})}  \ge \frac{V}{4} =   \frac{-\chi(S(Y))}{2} \ge \scl(\partial S(Y))$$
Hence each of these inequalities is actually an equality, and the lemma follows.
\end{proof}

We now show that condition (SB) is generic in a strong sense. Given $\hat{Y}$, we are interested
in the set of $Y$ with $\partial S(Y)$ reduced obtained by labeling the edges of $\hat{Y}$
by words of length at most $n$. For each $n$, this is a finite set, and we give it the uniform distribution.

\begin{proposition}\label{prop:generic_trivalent_fatgraph_extremal}  
For any combinatorial trivalent fatgraph $\hat{Y}$, if $Y$ is a random fatgraph over $F$
obtained by labeling the edges of $\hat{Y}$ by words of length $n$, then $S(Y)$ is extremal
for $\partial S(Y)$ and is certified by some extremal quasimorphism $\overline{h}_{\sigma_Y}$,
with probability $1-O(C(\hat{Y},F)^{-n})$ for some constant $C(\hat{Y},F)>1$.
\end{proposition}
\begin{proof}
The constant $C(\hat{Y},F)$ depends only on the number of vertices of $\hat{Y}$. We make use
of some elementary facts about random reduced strings in free groups.

If we label the edges of $\hat{Y}$ with random reduced words of 
length $n$, it is true that there may be some small amount of folding necessary in order
to obtain a fatgraph with $\partial S(Y)$ cyclically reduced. However, the expected amount of
letters to be folded is a constant independent of $n$, which is asymptotically insignificant, 
and may be safely disregarded here and elsewhere for simplicity.

Now consider some element $w$ of $\sigma_Y$ under 
some random labeling.  The fatgraph $Y$ over $F$ will satisfy condition (SB) with the desired
probability if the probability that $w^{-1}$ appears (as a subword) in $\partial S(Y)$ is 
$C^{-n}$, because the number of elements of $\sigma_Y$ is fixed (note that we are using the elementary
but useful fact in probability theory that the maximum probability of a conjunction of extremely rare events is well approximated by assuming the events are independent).

If $w^{-1}$ appears in $\partial S(Y)$, then at least half of it must appear as a subword 
of one of the edges of $Y$, so the probability that $w^{-1}$ appears in $\partial S(Y)$ 
is certainly smaller than the probability that the prefix or suffix of $w$ of length $n/2$ 
appears as a subword of an edge of $Y$.  Let $k$ denote the number of edges of $\hat{Y}$.  
The probability that a subword of length $n/2$ appears in a word of length $n$ is 
approximately $(n/2)\mathrm{rank}(F)^{-n/2}$, so, as we must consider each edge and its 
inverse, the probability that $w^{-1}$ appears is smaller than $2k(n/2)\mathrm{rank}(F)^{-n/2}$.  
By replacing $\mathrm{rank}(F)$ by a slightly smaller constant, we may disregard the $(n/2)$ 
multiplier, and the lemma is proved.
\end{proof}

\subsection{Higher valence fatgraphs}\label{higher_valence}

For fatgraphs with higher valence vertices, the construction of a candidate extremal
quasimorphism is significantly more delicate.

For $m\ge 3$ let $K_m$ be the complete graph on $m$ vertices. Label the vertices
$0,1,2,\cdots,m-1$. Define a weight $w_m$ on directed edges $(i,j)$ of $K_m$ by the formula
$w_m(i,i+k) = 3-(6k/m)$ where indices are taken mod $m$. 

\begin{lemma}\label{lem:weight_formula}
The function $w_m(i,i+k):=3-(6k/m)$ is the unique function on directed edges of $K_m$ with
the following properties:
\begin{enumerate}
\item{It is antisymmetric: $w_m(i,j)=-w_m(j,i)$.}
\item{It satisfies the inequality $|w_m(i,j)|\le 3-6/m$ for all distinct $i,j$.}
\item{For every distinct triple $i,j,k$, there is an equality 
$w_m(i,j)+w_m(j,k)+w_m(k,i)=\pm 3$
where the sign is positive if the natural cyclic order on $i,j,k$ is positive, and
negative otherwise.}
\item{It satisfies $w_m(i,i+1)=3-6/m$ for all $i$.}
\end{enumerate}
\end{lemma}
\begin{proof}
Only uniqueness is not obvious. If we think of $w_m$ as a simplicial $1$-cochain on the
underlying simplicial structure on the regular $m-1$ simplex then condition (3) determines
the coboundary $\delta w_m$, so $w_m$ is unique up to the coboundary of a function on vertices.
But condition (4) says this coboundary is zero.
\end{proof}

For $x$ a reduced word in $F$, parameterize $x$ proportional to arclength as the
interval $[-1,1]$, and let $x[-t,t]$ denote the smallest subword containing the
interval from $x(-t)$ to $x(t)$. Fix some small $\epsilon > 0$ and
define the {\em stack function} $S_x$ to be the following
integral of big counting functions:
$$S_x = \frac 1 {1-\epsilon} \int_\epsilon^1 C_{x[-t,t]} dt$$
The $\epsilon$ correction term ensures that the length of the shortest word in the
support of $S$ is at least $\epsilon|x|$. If $x$ is quite long, this word will also
be quite long, and ensure that there are no ``accidents'' in what follows. The
constant $\epsilon$ we need is of order $1/(\max_v |v|)$; we leave it implicit in
what follows, and in practice ignore it.
\begin{remark}
The function $S_x$ is actually a finite rational sum of ordinary big counting functions, 
since $x[-t,t]$ takes on only finitely many values. We can make it into a genuine
integral by first applying the (isometric) endomorphism $\varphi_m$ to $F$ which
takes every generator to its $m$th power, and then taking
$\lim_{m \to \infty} \varphi_m^* S_{\varphi_m(x)}$ in place of $S_x$.
However, this is superfluous for our purposes here.
\end{remark}

We are now in a position to define the quasimorphism $H_Y$.

\begin{definition}
Let $Y$ be a fatgraph over $F$, and suppose that every edge has length $\ge 2n$.
For each vertex $v$, denote the set of oriented subarcs in $\partial Y$ of length
$n$ ending at $v$ by $x_i(v)$, where the index $i$ runs from $0$ to $|v|-1$ and
the cyclic order of indices agrees with the cyclic order of edges at $v$. Denote
the inverse of $x_i(v)$ by $X_i(v)$.

Then define
$$H_Y = \sum_v \sum_{i,i+k \% |v|} (3-(6k/|v|)) (S_{x_i(v)X_{i+k}(v)} - S_{x_{i+k}(v)X_i(v)})$$
(note that the factor $3-(6k/|v|)$ is $w_{|v|}(i,i+k)$ from Lemma~\ref{lem:weight_formula}).
\end{definition}

Let $\sigma$ denote a word of the form $x_i(v)X_j(v)$ or its inverse. In other words,
the $\sigma$ are the words in the support of $H_Y$.
Now say that $Y$ satisfies condition (B) if, whenever some $\sigma[a,b]$ appears as a subword
of some other $\sigma'$, or some $\sigma[a,b]$ or its inverse appears twice in $\sigma$,
then $(b-a)$ is not too big --- explicitly,
$(b-a) < 6/4(\max_v |v|)$. Hereafter we denote $\delta:=6/4(\max_v |v|)$.

\begin{lemma}
Suppose $Y$ satisfies condition (B). Then $D(H_Y)\le 3$.
\end{lemma}
\begin{proof}
Condition (B) says that if two distinct $\sigma,\sigma'$ overlap a junction on
one side of a tripod, then $S_\sigma, S_{\sigma'}$ each
contributes at most $\delta$ to the defect. So we can assume that on at least
one side, there is a unique $\sigma=x_i(v)X_j(v)$ with a subword of definite size
that overlaps a junction. Again, without loss of generality, we can assume
that the junction is at $\sigma(t)$ where $t\in [-1+\delta,1-\delta]$.
By condition (B), if $\sigma'$ on another side has a subword of definite size that
overlaps the junction, it either contributes
at most $\delta$, or else we must have $\sigma'=x_k(v)X_i(v)$ or $\sigma'=x_j(v)X_k(v)$.
So the only case to consider is when the three incoming directed edges at the junction
are suffixes of $x_i(v),x_j(v),x_k(v)$ of length $1\ge s \ge t \ge u \ge 0$ respectively.
But in this case the total contribution to the defect is $u(w_{|v|}(i,j)+w_{|v|}(j,k)+
w_{|v|}(k,i)) + (t-u)w_{|v|}(i,j)$. Since $|w_{|v|}(i,j)+w_{|v|}(j,k)+
w_{|v|}(k,i)|=3$ and $|w_{|v|}(i,j)|<3$, this defect is $\le 3$, as claimed.
\end{proof}

\begin{theorem}[Random fatgraph theorem]\label{fatgraph_theorem}
For any combinatorial fatgraph $\hat{Y}$, if $Y$ is a random fatgraph over $F$ obtained by
labeling the edges of $\hat{Y}$ by words of length $n$, then $S(Y)$ is extremal 
for $\partial S(Y)$ and is certified by the extremal quasimorphism $\overline{H}_Y$, with
probability $1-O(C(\hat{Y},F)^{-n})$ for some constant $C(\hat{Y},F)>1$.
\end{theorem}
\begin{proof}
The argument is a minor variation on the arguments above, so we just give a sketch of
the idea.

It suffices to show that a random $Y$ satisfies condition (B) with probability
$1-O(C^{-n})$ for some $C$. But this is obvious, since the $x_i(v)$ are independent,
and for any constant $\kappa > 0$, two random words in $F$ of length $n$ do not have 
overlapping segments of length bigger than $\kappa n$, and
a random word of length $n$ does not have a segment of length bigger than $\kappa n$ that
appears twice, in either case with probability $1-O(C^{-n})$.
\end{proof}

\begin{remark}
Since $\chi(S(Y)) \in \Z$, the boundary $\partial S(Y)$ satisfies $\scl(\partial S(Y))\in\frac{1}{2}\Z$.
On the other hand, Theorem~\ref{fatgraph_theorem} does \emph{not} imply anything about 
the structure of $\scl$ for generic \emph{chains} 
of a particular length.  A random homologically trivial word (or chain) in a hyperbolic group
of length $n$ has $\scl$ of size $O(n/\log{n})$ (see \cite{Calegari_Maher}), so a random
homologically trivial word of length $n$ {\em conditioned to have genus bounded by some constant},
will be very unusual.

In fact, computer experiments suggest that the expected denominator of $\scl(w)$ is a proper
function of the length of a (random) word $w$.
\end{remark}

There are only finitely many distinct combinatorial fatgraphs with a given Euler characteristic,
so if we specialize $\hat{Y}$ to have a single boundary component (recall this depends only 
on the combinatorics of $\hat{Y}$ and not on any particular immersion), then 
then we see that for any integer $m$ there is a constant $C$ depending on $m$ so that a random word of 
length $n$ conditioned to have commutator length at most $m$ has commutator length 
exactly $m$ and $\scl = m-1/2$, with probability $1-O(C^{-n})$.

\subsection{Experimental data}\label{experimental_data_subsection}

Our main purpose in this section is to give an experimental check of our results and 
to estimate the constants $C(\hat{Y},F)$.  However, it is worth mentioning that vertex 
quasimorphisms provide quickly verifiable rigorous (lower) bounds on $\scl$.  

\subsubsection{Fast rigorous lower bounds on $\scl$}

Although not every chain admits an extremal surface which is {\em certified} by a vertex 
quasimorphism, it happens much more frequently that a vertex quasimorphism certifies good lower
bounds on $\scl$. For example, if $Y$ is not trivalent, a quasimorphism of the form
$\bar{h}_{\sigma_Y}$ will never be extremal; but if the average valence of $Y$ is close to $3$, 
the lower bound one obtains might be quite good. 

Because verifying condition (B) requires only checking the (non)-existence of certain words as 
subwords of the boundary $\partial S(Y)$, plus a small cancellation condition, it is possible 
to certify the defect of a vertex quasimorphism in polynomial time. This compares favorably to
the problem of computing the defect of an arbitrary linear combination of big counting 
quasimorphisms (or even a single big counting quasimorphism) for which the best known algorithms
are exponential. 

\begin{example}
It is rare for (short) words or chains to admit extremal trivalent fatgraphs. A cyclic word
is {\em alternating} if it contains no $a^{\pm 2}$ or $b^{\pm 2}$ substring; for example,
$baBABAbaBabA$ is alternating, with $\scl=5/6$. An extremal fatgraph for an alternating
word necessarily has all vertices of even valence, since the edge labels at each vertex must alternate
between one of $a^\pm$ and one of $b^\pm$.
\end{example}

\subsubsection{Experimental calculation of constants $C(\hat{Y},F)$}

\begin{figure}
\centering
\begin{minipage}[htpb]{0.49\linewidth}
\flushright
\labellist
\small\hair 2pt
\pinlabel $4$ at 20 4
\pinlabel $11$ at 90 4
\pinlabel $1$ at 5 20
\pinlabel $12$ at 3 130
\pinlabel $-\log(P(\text{fail}))$ at -15 75
\pinlabel {label length} at 54 2
\endlabellist
\includegraphics[scale=1.3]{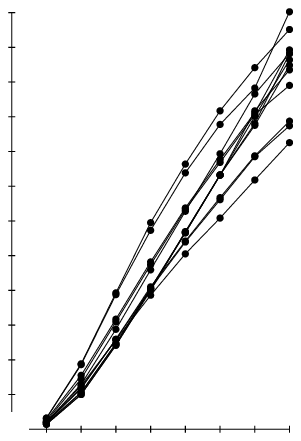}
\end{minipage}
\begin{minipage}[h]{0.49\linewidth}
\flushleft
\begin{tabular}{ccc}
\includegraphics[scale=0.09]{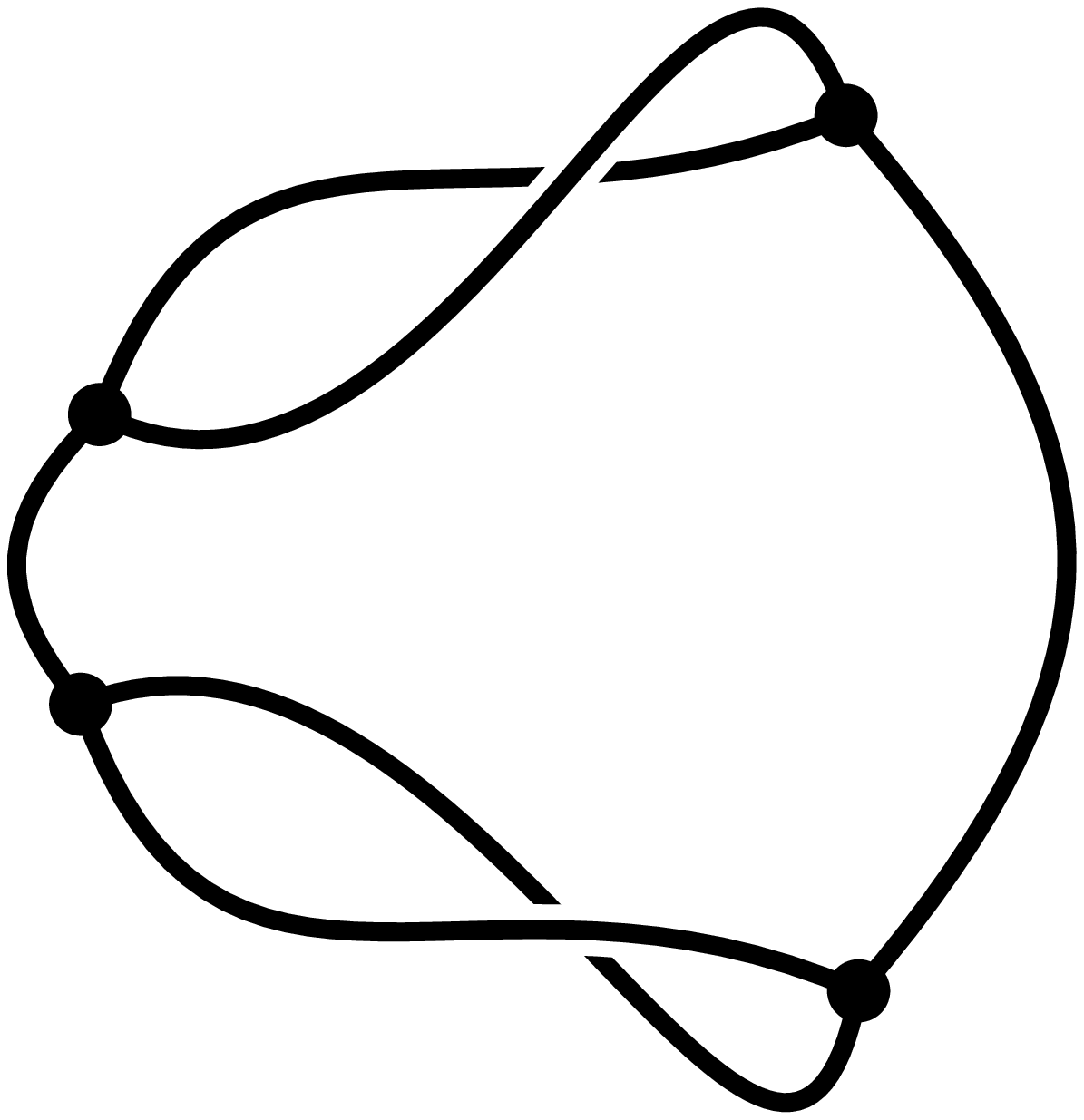} & \includegraphics[scale=0.09]{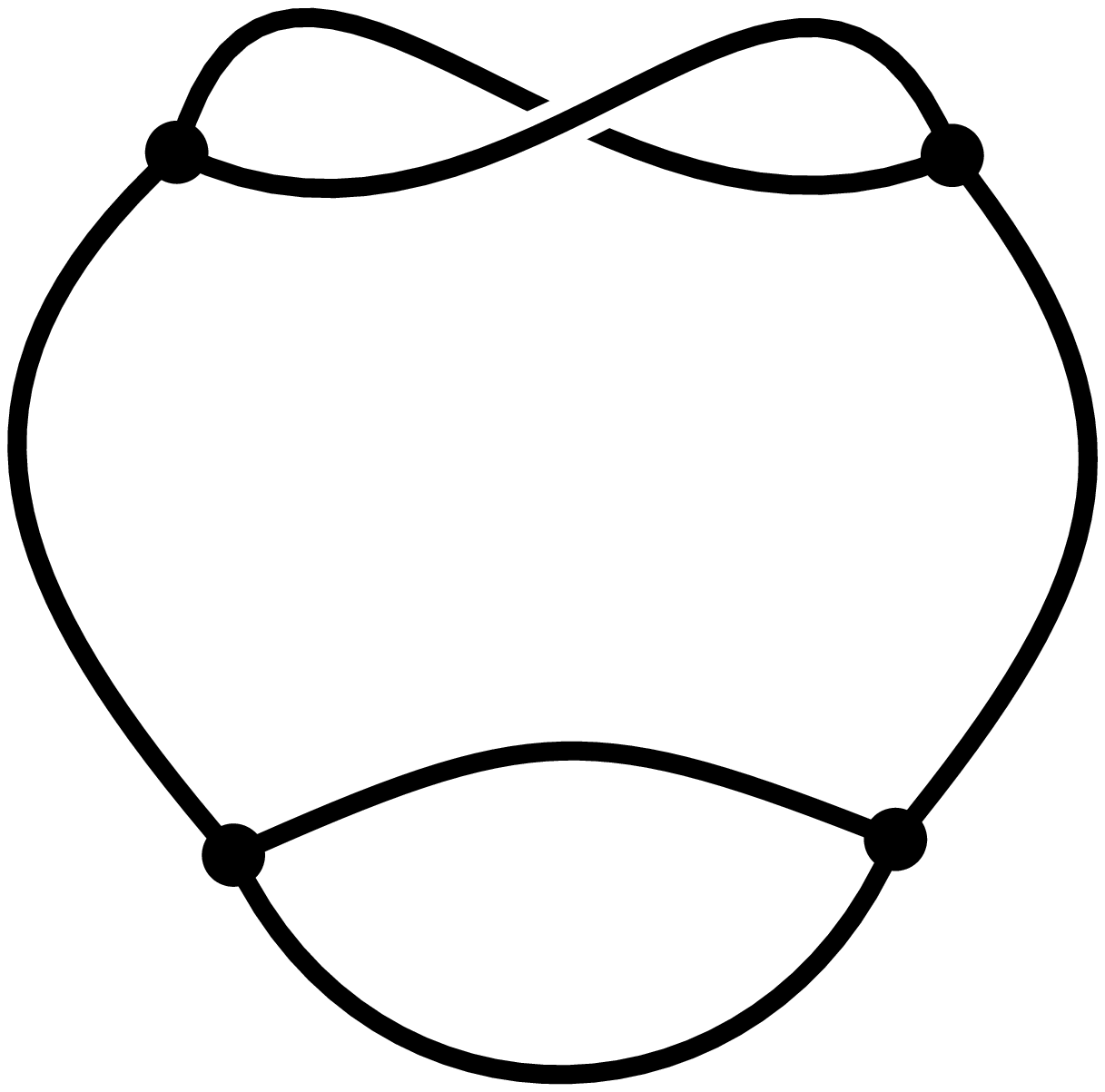} &  \includegraphics[scale=0.09]{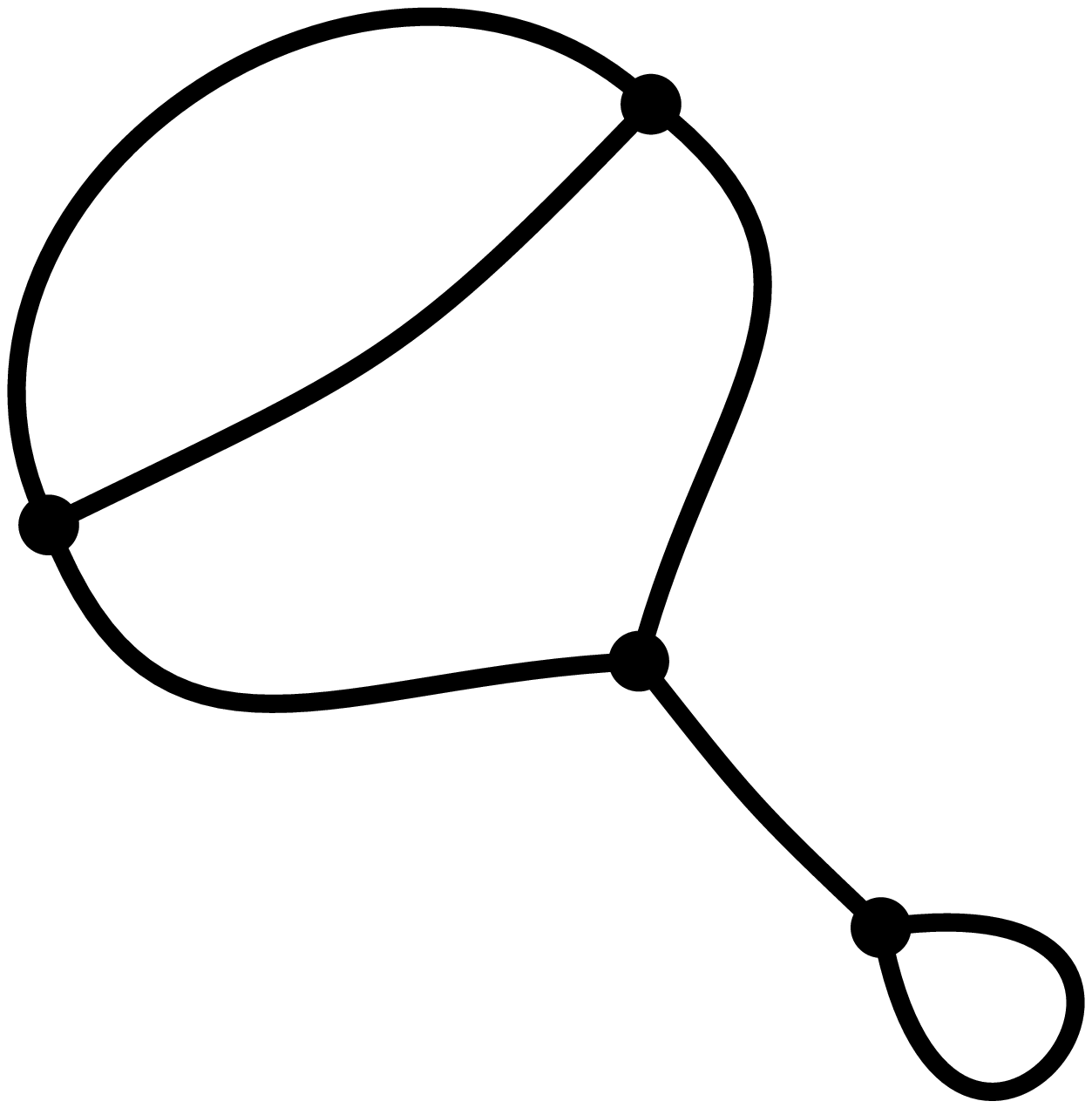}\\
\includegraphics[scale=0.09]{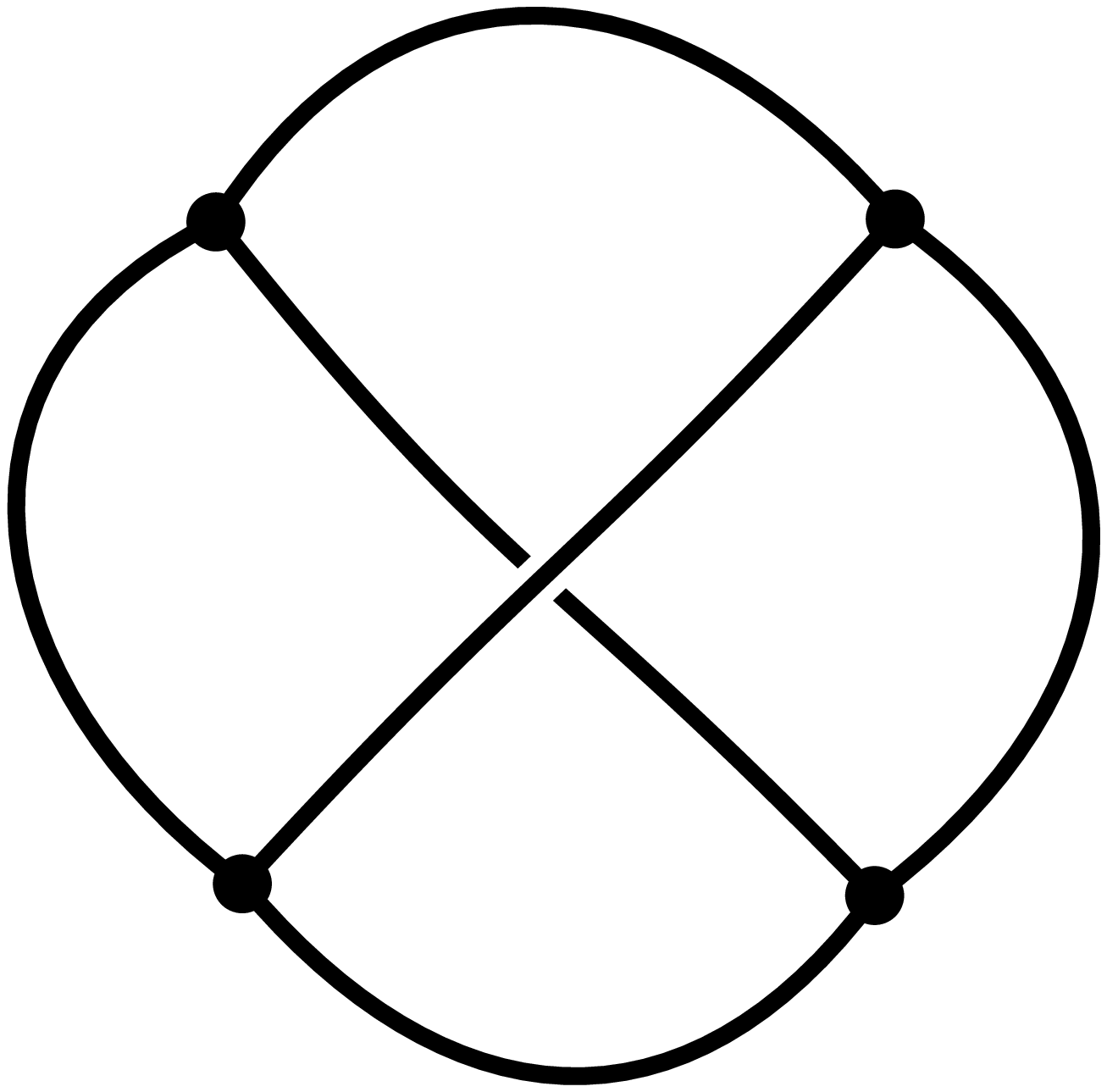} & \includegraphics[scale=0.09]{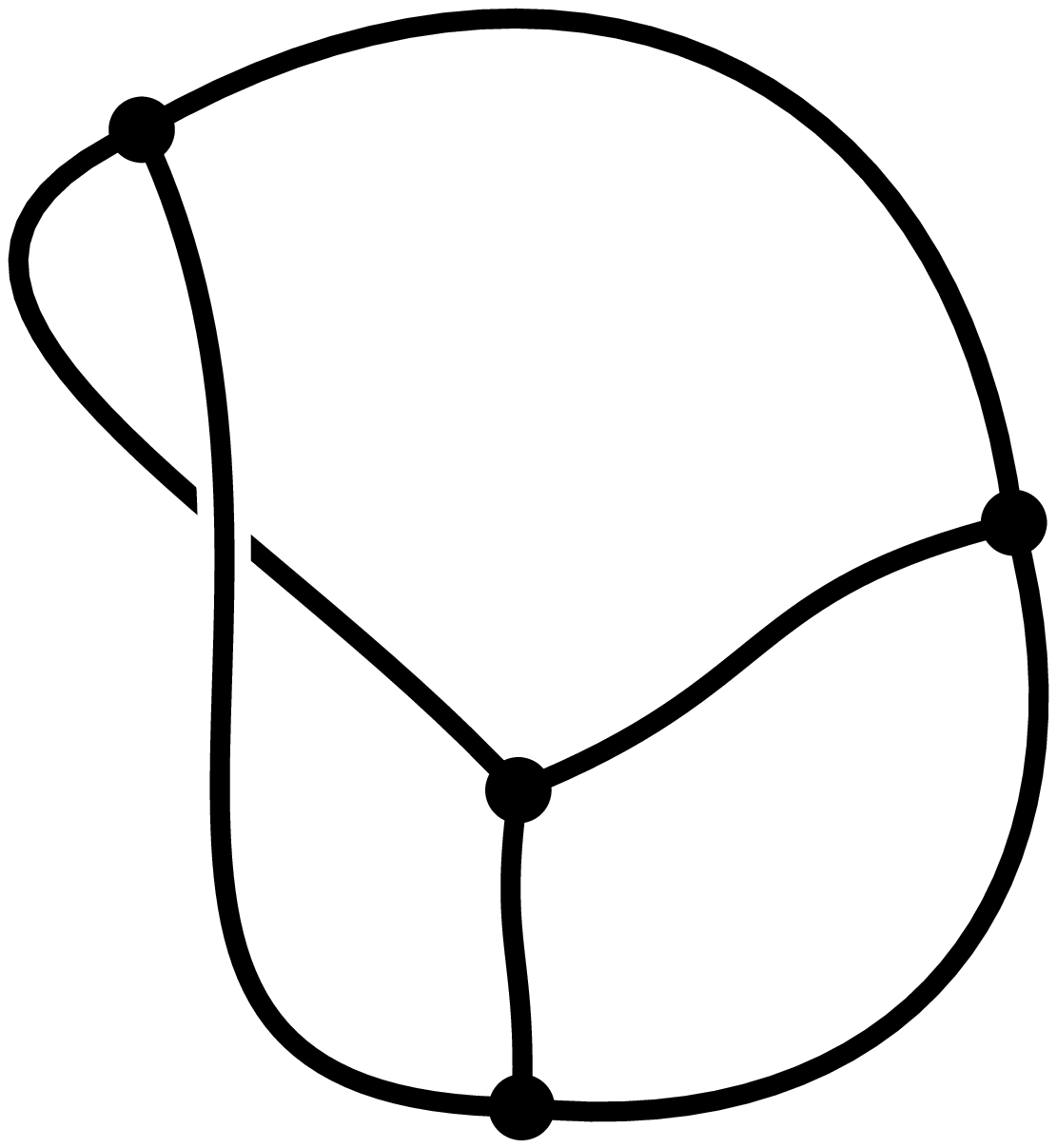} &  \includegraphics[scale=0.09]{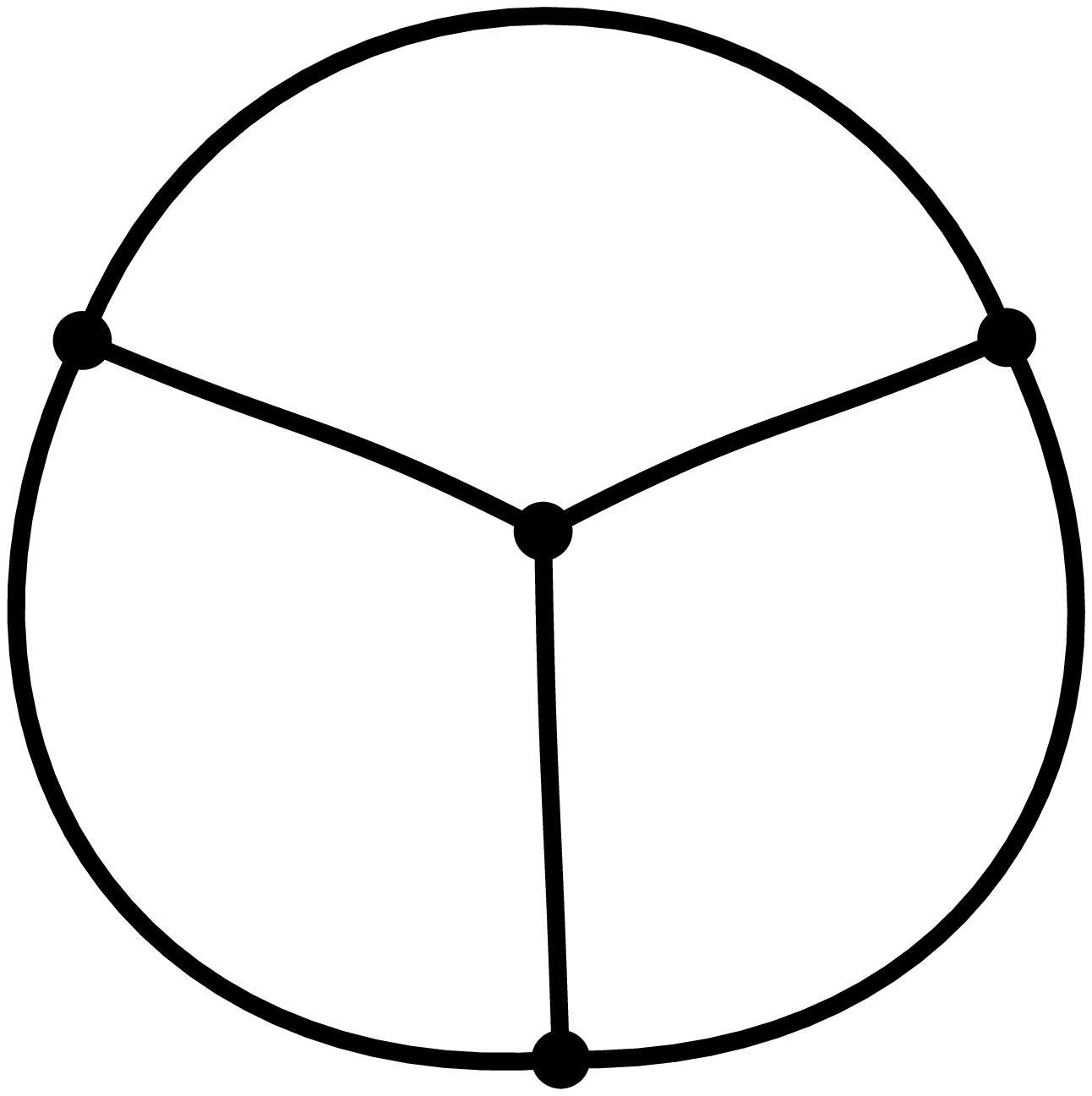} \\
\includegraphics[scale=0.09]{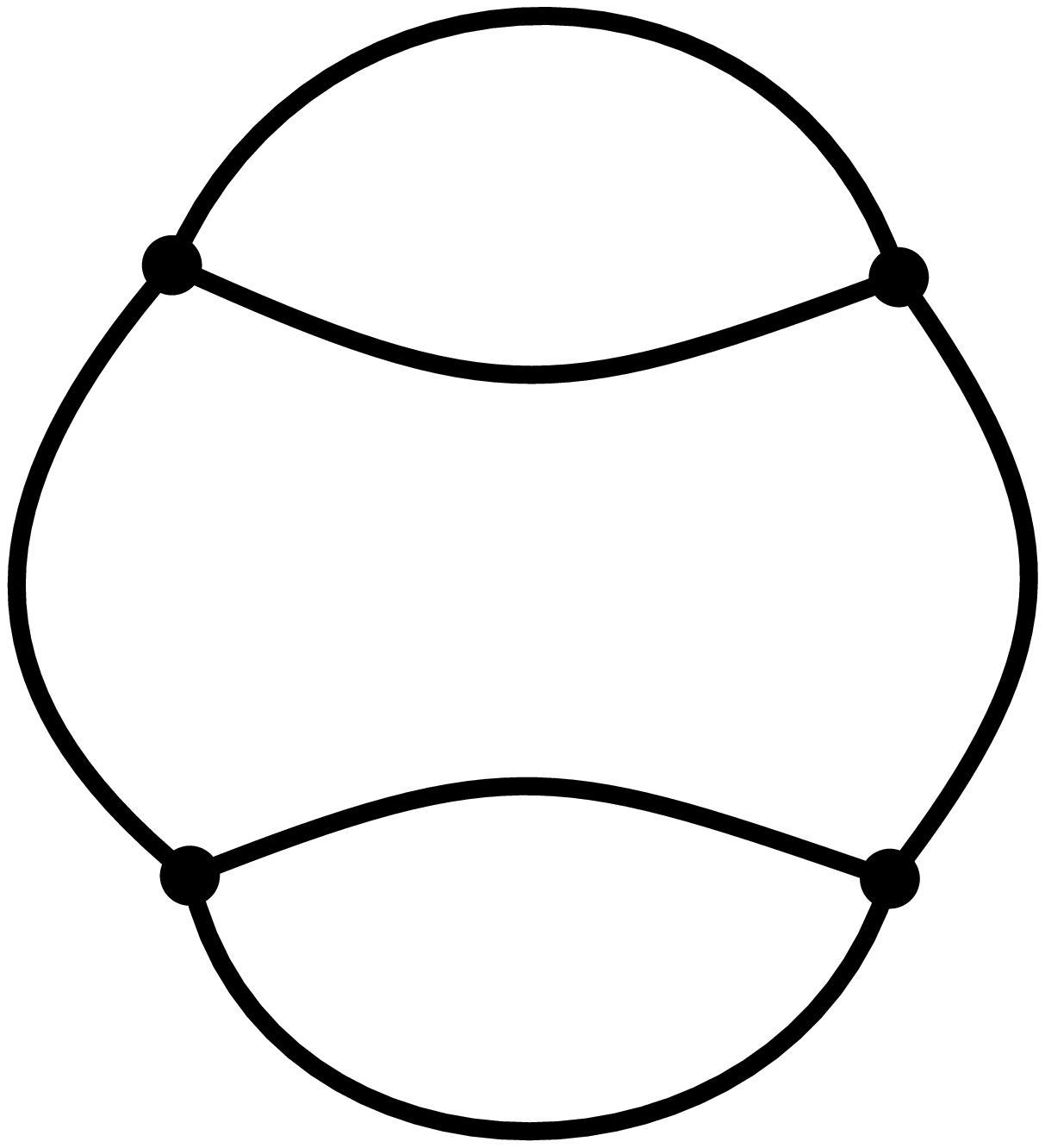} & \includegraphics[scale=0.09]{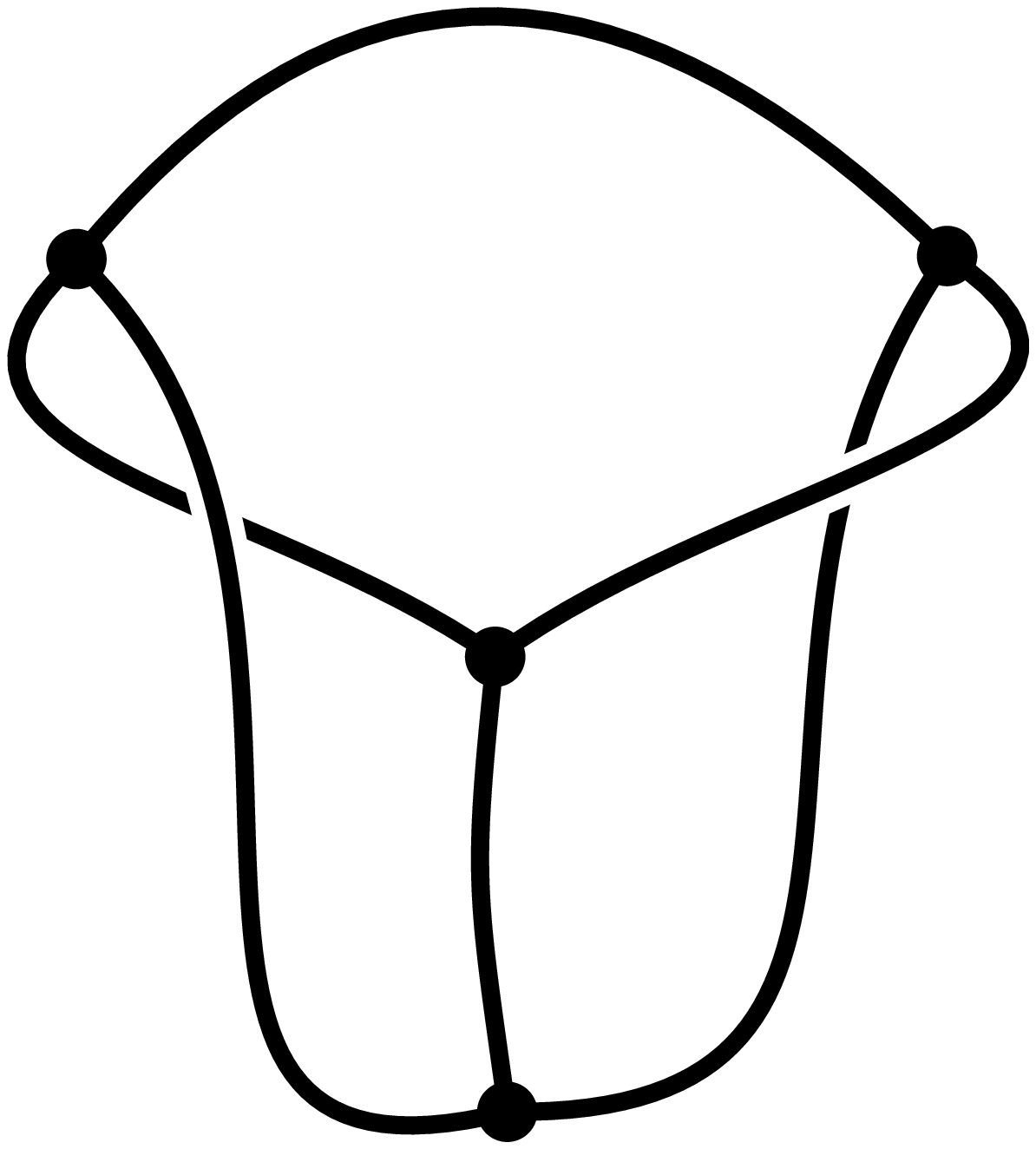} &  \includegraphics[scale=0.09]{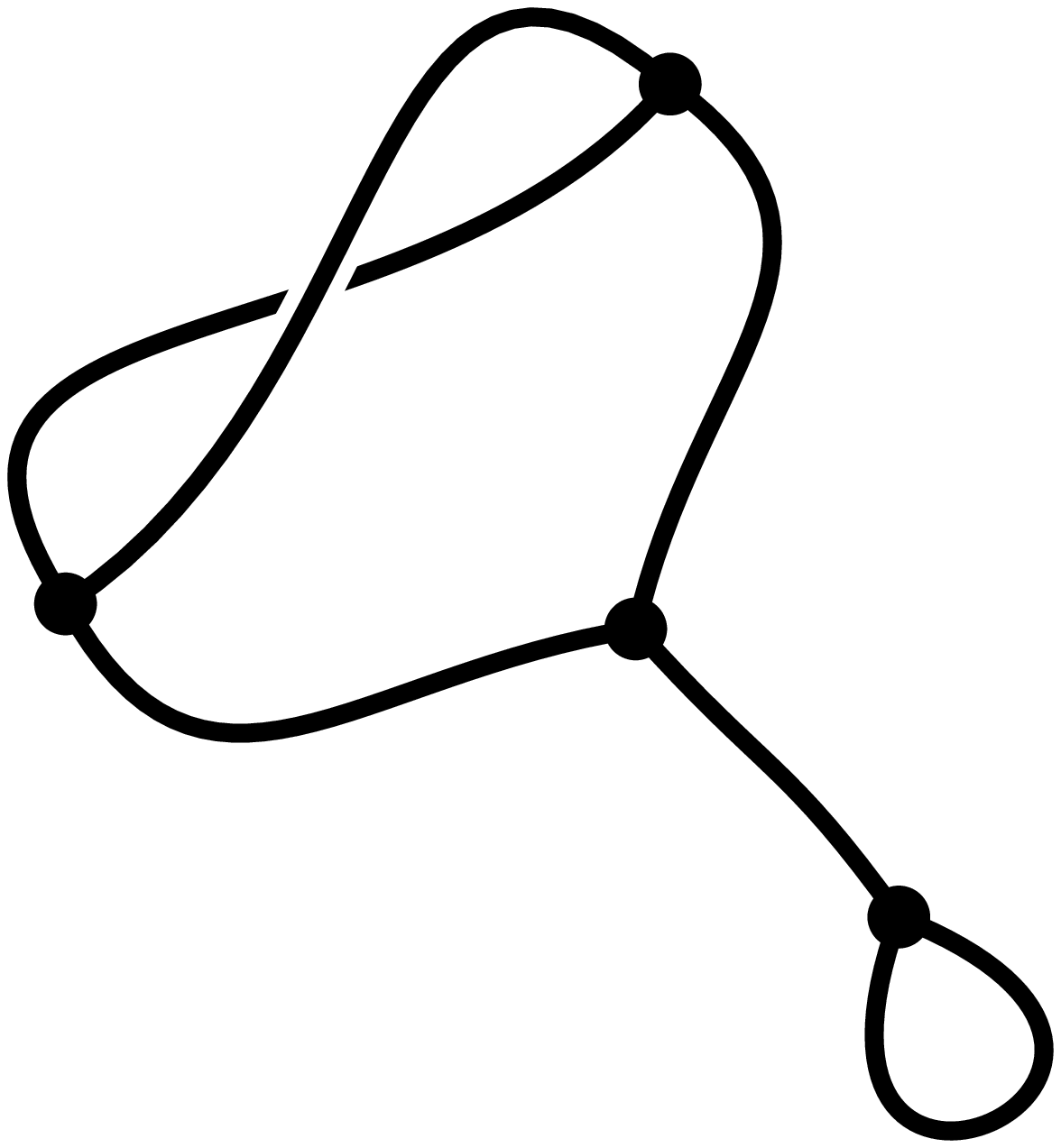} \\ \includegraphics[scale=0.09]{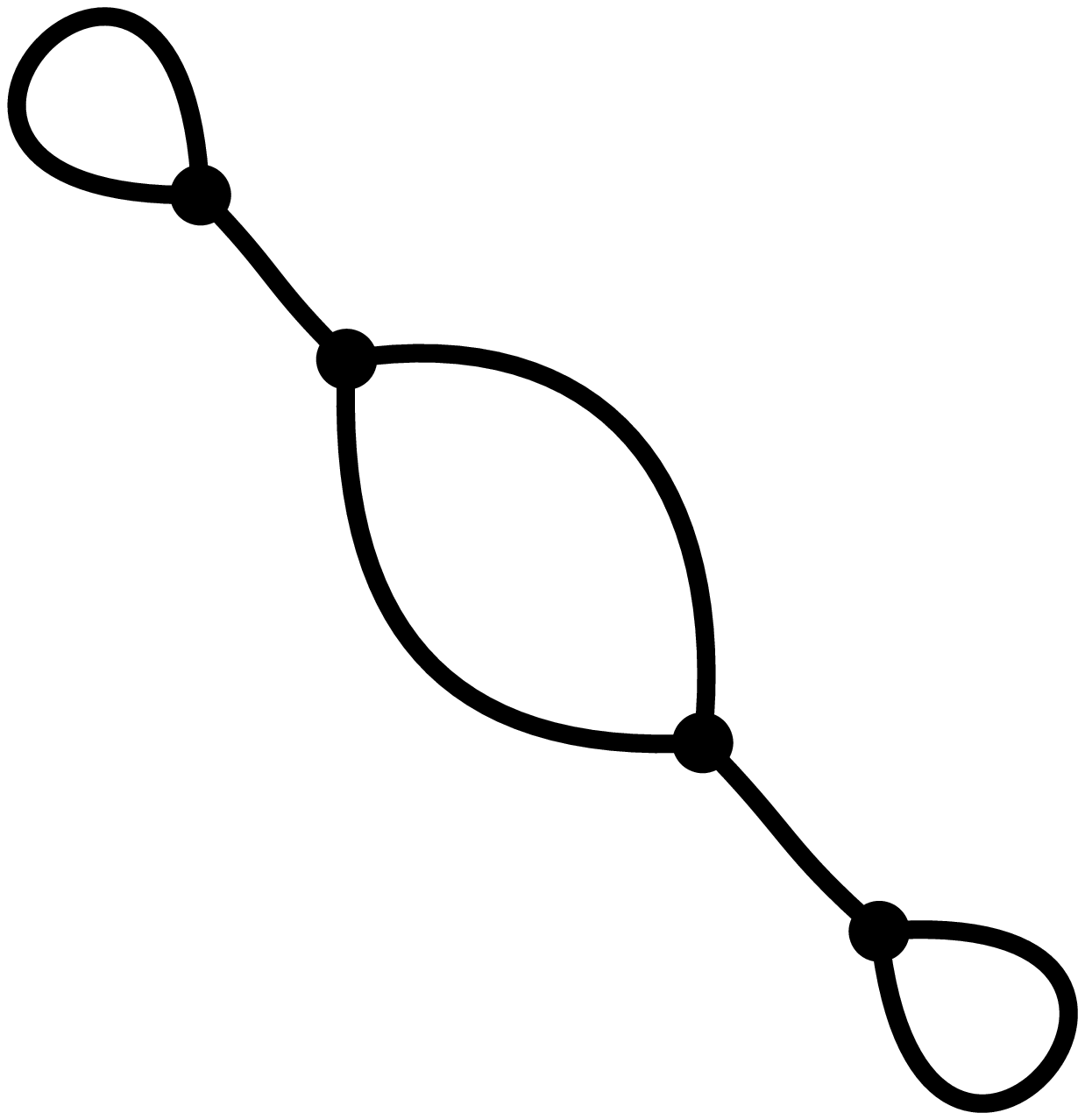} & \includegraphics[scale=0.09]{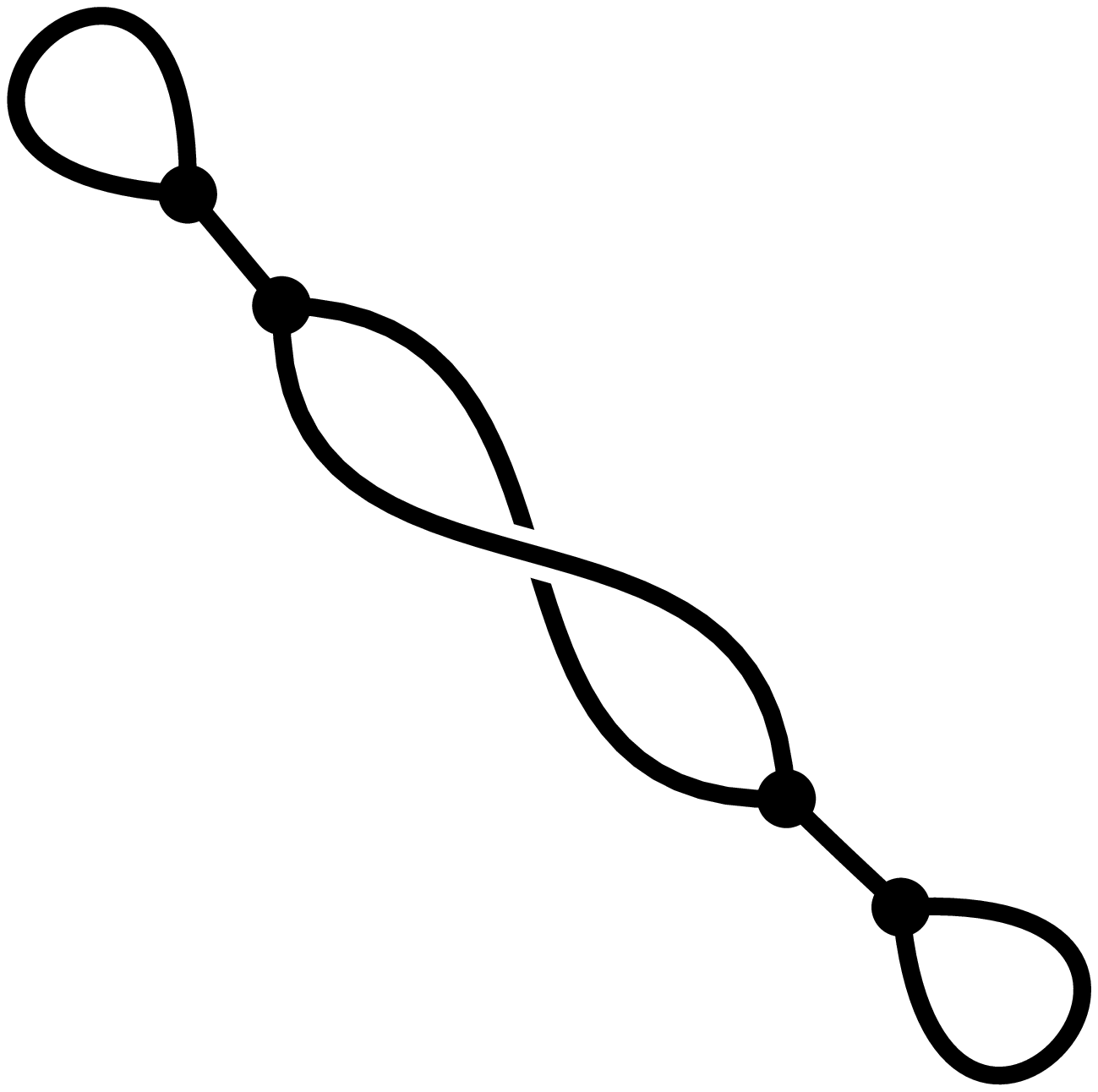} & \includegraphics[scale=0.09]{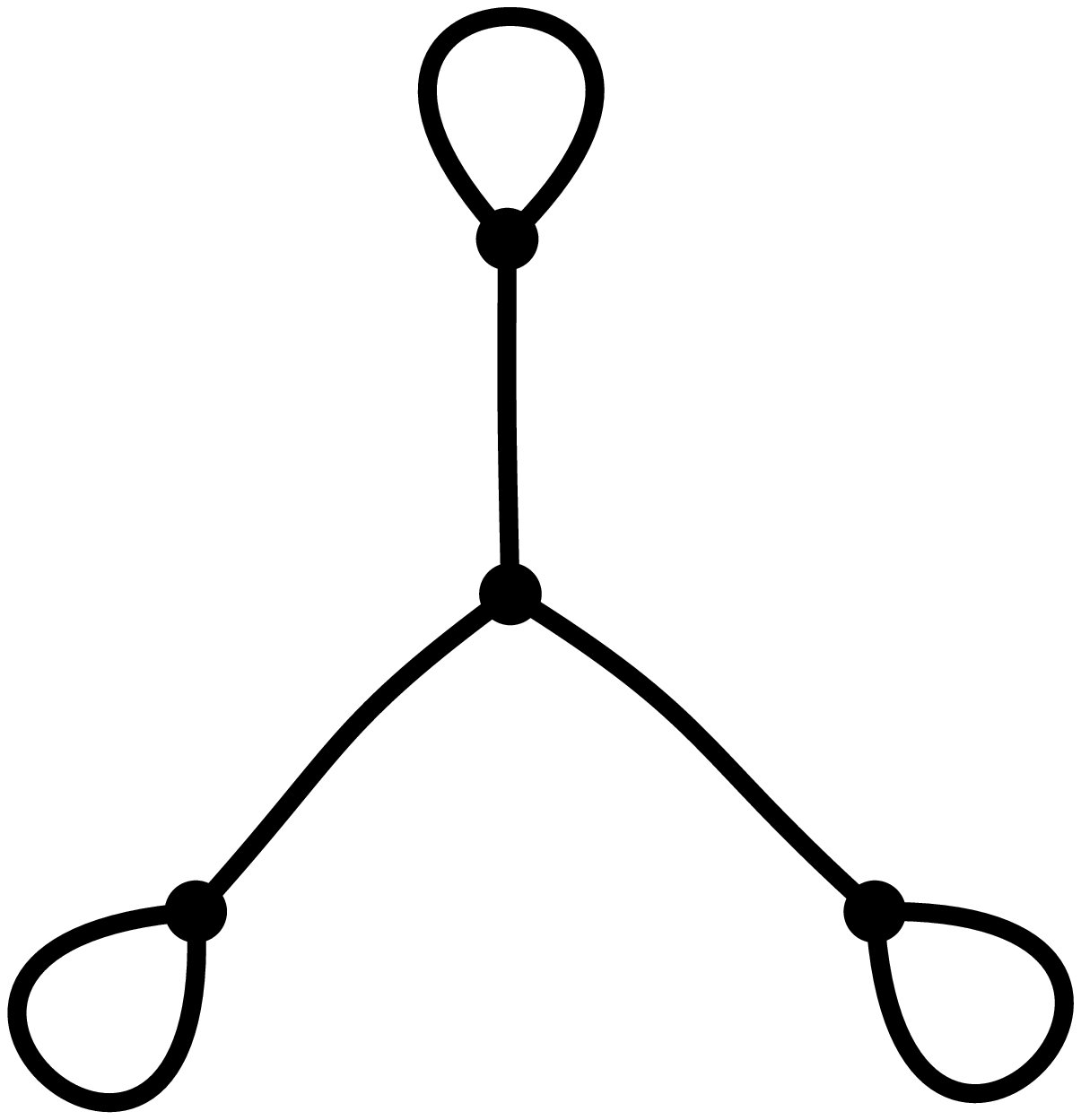} \\
\end{tabular}
\end{minipage}
\caption{A plot of $-\log$ of the failure rate for random labelings of lengths 
between $4$ and $11$, plotted for each trivalent fatgraph with four vertices.  
Each dot represents $500,000$ trials.  The fatgraphs themselves are arranged 
left to right, top to bottom in decreasing order of $-\log$ of failure at 
length 11, so the tripod in the lower right is the ``hardest'' to find vertex 
quasimorphisms for.  The pictures were created using {\tt wallop} \cite{wallop}.}\label{fig:exp_data}
\end{figure}

While {\em certifying} a vertex quasimorphism is easy, {\em finding} one is much harder.
To verify our asymptotic results, we can be content with breaking the edges of the fatgraph
into uniform pieces and checking whether condition (B) is satisfied.  However, for a given fatgraph,
it might be the case (and usually is) that while a naive assignment of words for $H_Y$ fails, a
more careful choice succeeds.  To check whether there is \emph{any} vertex quasimorphism is 
(naively) exponential, and this makes large experiments difficult.

However for {\em trivalent} fatgraphs, condition (SB) on $\sigma_Y$ is much simpler.  In particular, 
whether or not a collection of words satisfies (SB) depends only on the (local) no-overlap condition,
plus the ``constant'' condition of certain words not appearing in $\partial S(Y)$.  This makes
this a priori infeasible problem of checking whether there is any vertex quasimorphism for a 
particular fatgraph possible with the use of a ``meet-in-the-middle'' time-space tradeoff.

Using this method, we can experimentally estimate the best possible constants 
$C(\hat{Y},F)$, at least in the case of trivalent $\hat{Y}$. Figure~\ref{fig:exp_data}
shows some data on the likelihood that a random labeling of a trivalent fatgraph 
with four vertices admits a vertex quasimorphism.  The linear dependence of $-\log(P(\text{fail}))$
on label length is evident. We can calculate a best fit 
slope and $y$-intercept for these lines, which gives a best fit line of 
$1.47336 n - 1.42772$, or equivalently, $P(\textrm{success}) = 1 - 4.16918(4.36387)^{-n}$.  
Note that the lower right graph is the least likely to admit a vertex quasimorphism; 
this is heuristically reasonable, since self-loops at vertices handicap the graph by forcing 
a shorter length on some words in $\sigma_Y$.  A best fit for this line yields 
$P(\textrm{success}) \ge 1 - 82.3971(3.19827)^{-n}$.

\subsubsection{Using homomorphisms to improve success rate}

When a particular labeling $Y$ does not admit a vertex quasimorphism, it might still be
possible to find an extremal quasimorphism by applying a homomorphism 
$\phi$ to $Y$. If (the folded fatgraph) $\phi(Y)$ admits an extremal vertex quasimorphism 
$\overline{H}_{\phi(Y)}$, and folding does not change the Euler characteristic of the fatgraph, 
then the quasimorphism $\phi^*\overline{H}_{\phi(Y)}$ is extremal for $\partial S(Y)$.  

\begin{figure}
\centering
\begin{minipage}[htpb]{0.49\linewidth}
\flushright
\labellist
\small\hair 2pt
\pinlabel $4$ at 20 4
\pinlabel $11$ at 90 4
\pinlabel $1$ at 5 20
\pinlabel $12$ at 3 130
\pinlabel $-\log(P(\text{fail}))$ at -15 75
\pinlabel {label length} at 54 2
\endlabellist
\includegraphics[scale=1.3]{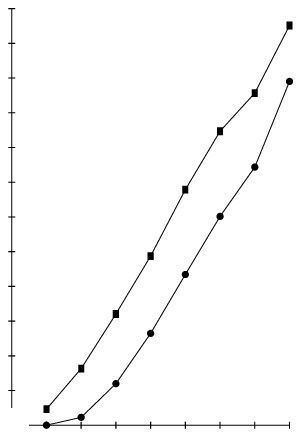}
\end{minipage}
\begin{minipage}[h]{0.36\linewidth}
\flushleft
\includegraphics[scale=0.2]{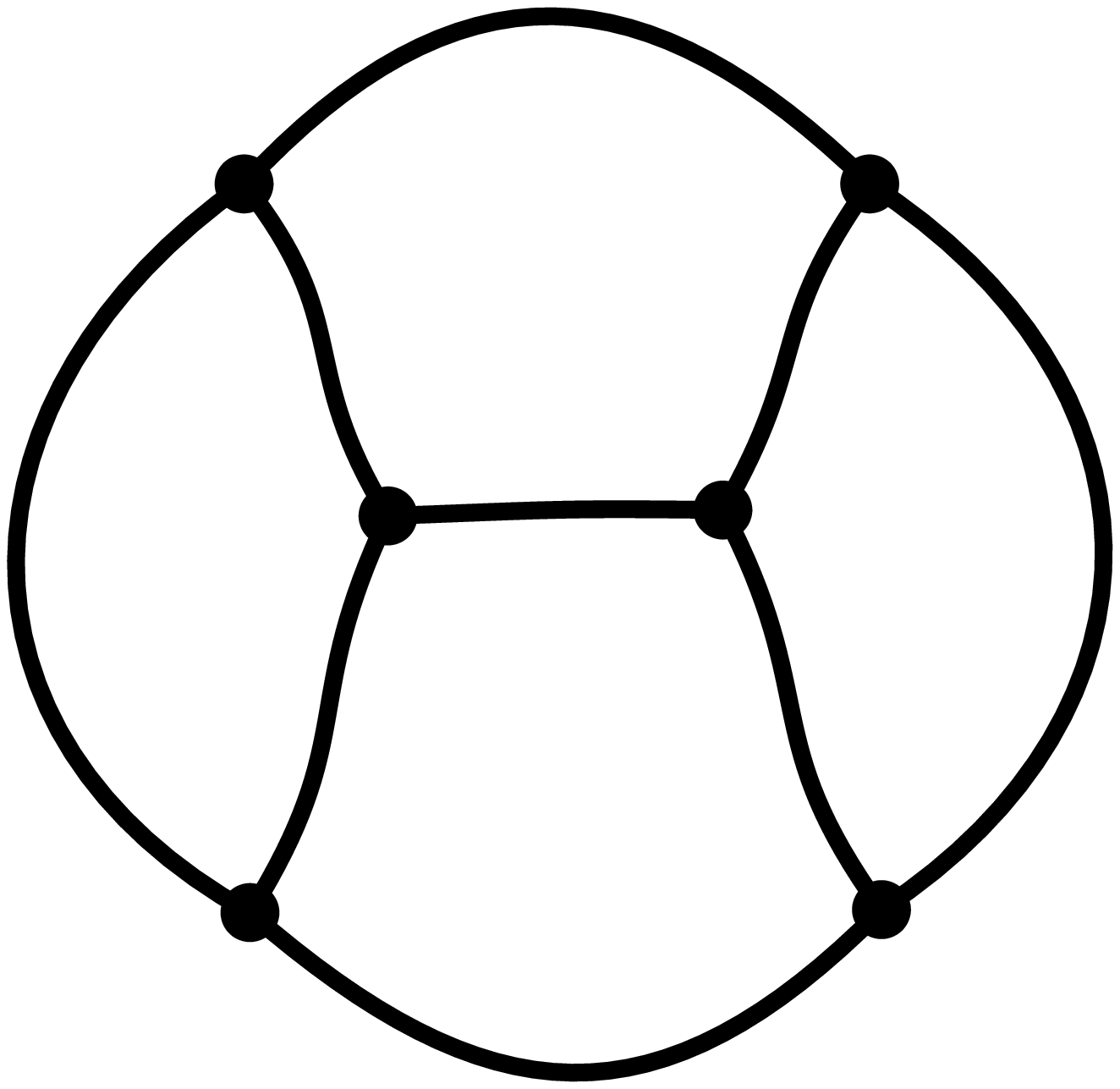}
\end{minipage}
\caption{A plot of $-\log$ of the failure rate for labelings of the fatgraph 
(circles), and $-\log$ of the failure rate after acting by many random 
homomorphisms (squares).}\label{fig:hom_data}
\end{figure}

Because the edges of $\phi(Y)$ are no longer random (and in particular, distinct edge labels
will necessarily share long common subwords), it is not clear 
that applying a homomorphism will affect our success rate.  In fact, it turns out to help
significantly, especially for shorter labelings.  Figure~\ref{fig:hom_data} shows $-\log$ of the 
failure rate for a particular fatgraph compared with $-\log$ of the failure rate 
after applying many random homomorphisms.  We decrease the probability of failure by a factor
of about $5$.  Interestingly, changing the length of the homomorphism 
or the number of homomorphisms that we try does not seem to significantly alter our success 
with this procedure.

\section{Acknowledgments}
Danny Calegari was supported by NSF grant DMS 1005246. We would like to thank Mladen Bestvina and
Geoff Mess for some useful conversations about this material.

\end{document}